\documentclass[a4paper,10pt]{article}
\usepackage[english]{babel}
\usepackage[latin1]{inputenc}
\usepackage[T1]{fontenc}
\usepackage{amsmath}
\usepackage{dsfont}
\usepackage{amssymb}
\usepackage{multirow}
\usepackage{alltt}
\usepackage{graphicx}
\usepackage{bbm}
\usepackage{amsthm}
\usepackage{mathrsfs}
\usepackage{textcomp}
\usepackage{amsrefs}

\usepackage{fullpage}
\usepackage[usenames,dvipsnames]{color}

\definecolor{darkblue}{RGB}{0,0,170}
\definecolor{brickred}{RGB}{200,0,0}
\usepackage[pdfauthor={ }, hyperindex]{hyperref}
\hypersetup{colorlinks=true, linkcolor=darkblue, citecolor=brickred}

\newcommand{\R}{\mathbb{R}}
\newcommand{\N}{\mathbb{N}}

\newcommand{\T}{\mathcal{T}}

\newcommand{\eps}{\varepsilon}
\newcommand{\dist}{\hbox{dist}}
\newcommand{\Ds}{{\left(\left.-\lapl\right|_\Omega\right)}^s}
\newcommand{\Dums}{{\left(\left.-\lapl\right|_\Omega\right)}^{1-s}}
\newcommand{\Dms}{{\left(\left.-\lapl\right|_\Omega\right)}^{-s}}
\newcommand{\Dsmu}{{\left(\left.-\lapl\right|_\Omega\right)}^{s-1}}
\newcommand{\rest}{{\left(\left.-\lapl\right|_\Omega\right)}}
\newcommand{\trest}{{\left.-t\lapl\right|_\Omega}}
\newcommand{\super}{\overline}

\newcommand{\lapl}{\triangle}
\newcommand{\grad}{\triangledown\!}

\newtheorem{defi}{Definition}[]

\newtheorem{theo}[defi]{Theorem}
\newtheorem{prop}[defi]{Proposition}
\newtheorem{lem}[defi]{Lemma}
\newtheorem{rmk}[defi]{Remark}
\newtheorem{cor}[defi]{Corollary}

\title{\bf\scshape Nonhomogeneous boundary conditions for the spectral fractional Laplacian}
\author{\scshape nicola abatangelo, louis dupaigne}
\date{ }

\begin{document}

\maketitle

\begin{abstract}
We present a construction of harmonic functions on bounded domains for the {\it spectral fractional Laplacian} operator and we classify them in terms of their divergent profile at the boundary. This is used to establish and solve boundary value problems associated with nonhomogeneous boundary conditions. 
We provide a weak-$L^1$ theory to show how
problems with measure data at the boundary and inside the domain are well-posed. We study linear and semilinear problems, performing a sub- and supersolution method, and we finally show the existence of {\it large solutions} for some power-like nonlinearities. 
\end{abstract}

{\let\thefootnote\relax

}

\setcounter{footnote}{0}

\section{Introduction}
Given a bounded domain $\Omega$ of the Euclidean space $\R^N$, 
the {\it spectral fractional Laplacian} operator $\Ds,\ s\in(0,1),$ is classically defined as 
a fractional power of the Laplacian with homogeneous Dirichlet 
boundary conditions, 
seen as a self-adjoint operator in the Lebesgue space $L^2(\Omega)$, see \eqref{As} below.
This provides a nonlocal operator of elliptic type with {\it homogeneous} boundary conditions.
Recent bibliography on this operator can be found e.g. in Bonforte, Sire and Vazquez  \cite{bsv}, Grubb \cite{grubb}, Caffarelli and Stinga \cite{caffarelli-stinga}, Servadei and Valdinoci \cite{servadei-valdinoci}.

One aspect of the theory is however left unanswered:
the formulation of natural {\it nonhomogeneous} boundary conditions. A first attempt can be found in the work of Dhifli, M\^{a}agli and Zribi \cite{dhifli}. 
The investigations that have resulted in the present paper turn out, we hope, to shed some further light on this question. 
We provide a weak formulation, which is well-posed in the sense of Hadamard, for linear problems of the form  
\begin{equation}\label{main}
\left\{\begin{aligned}
\Ds u &=\ \mu & \hbox{ in }\Omega, \\
\frac{u}{h_1} &=\ \zeta & \hbox{ on }\partial\Omega,
\end{aligned}\right.
\end{equation}
where $h_1$ is a reference function, see \eqref{h1} below, with prescribed singular behaviour at the boundary. Namely, $h_1$ is bounded above and below by constant multiples of $\delta^{-(2-2s)}$, where
\[
\delta(x):=\dist(x,\partial\Omega)
\] 
is the distance to the boundary and the left-hand side of the boundary condition must be understood as 
a limit as $\delta$ converges to zero. 
In other words, unlike the classical Dirichlet problem for the Laplace operator, 
nonhomogeneous boundary conditions must be singular. 
In addition, if the data $\mu, \zeta$ are smooth, the solution blows up at the fixed rate $\delta^{-(2-2s)}$. 
This is very similar indeed to the theory of nonhomogeneous boundary conditions for the classical (sometimes called "restricted") fractional Laplacian - although in that case the blow-up rate is of order $\delta^{-(1-s)}$ -
as analysed from different perspectives by Grubb \cite{grubb0} and the first author \cite{a1} of the present note.
In fact, for the special case of positive $s$-harmonic functions, that is when $\mu=0$, 
the singular boundary condition was already identified in previous works 
emphasizing the probabilistic and potential theoretic aspects of the problem: 
see e.g. Bogdan, Byczkowski, Kulczycki, Ryznar, Song and Vondra{\v{c}}ek \cite{bbkr} 
for an explicit example in the framework of the classical fractional Laplacian, 
as well as Song and Vondra\v{c}ek \cite{song-vondra}, 
Glover, Pop-Stojanovic, Rao, \v{S}iki\'c, Song and Vondra\v{c}ek \cite{gprssv} and Song \cite{song}
for the spectral fractional Laplacian. 

Turning to nonlinear problems, even more singular boundary conditions arise: 
in the above system, if $\mu=-u^p$ for suitable values of $p$, one may choose $\zeta=+\infty$, 
in the sense that the solution $u$ will blow up at a higher rate
with respect to $\delta^{-(2-2s)}$ and 
controlled by the (scale-invariant) one $\delta^{-2s/(p-1)}$. 
Note that the value $\zeta=+\infty$ is not admissible for linear problems.
This was already observed by the first author in the context of the fractional Laplacian, see \cite{a2},
and this is what we prove here for the spectral fractional Laplacian. 
Interestingly, the range of admissible exponents $p$ is different according to which operator one works with.

\subsection*{Main results} 
For clarity, we list here the definitions and statements that we use, with reference to the sections of the paper where the proofs can be found. First recall the definition of the spectral fractional Laplacian:
\begin{defi}
Let $\Omega\subset\R^N$ a bounded 
domain and let ${\{\varphi_j\}}_{j\in\N}$ 
be a Hilbert basis of $L^2(\Omega)$ consisting of eigenfunctions of the Dirichlet Laplacian $\left.-\lapl\right\vert_\Omega$, associated to the eigenvalues $\lambda_j$, $j\in\N$, i.e.\footnote{See Brezis \cite[Theorem 9.31]{brezis}.} $\varphi_j\in H^1_0(\Omega)\cap C^\infty(\Omega)$ and $-\lapl\varphi_j=\lambda_j\varphi_j$ in $\Omega$.
Given $s\in(0,1)$, consider the Hilbert space\footnote{When $\Omega$ is smooth, $H(2s)$ coincides with the Sobolev space $H^{2s}(\Omega)$ if $0<s<1/4$, $H_{00}^{s}(\Omega)$ if $s\in\{1/4, 3/4\}$, $H_0^{2s}(\Omega)$ otherwise; see Lions and Magenes \cite[Theorems 11.6 and 11.7 pp. 70--72]{lions-magenes}.} 
\[
H(2s):=\left\{v=\sum_{j=1}^\infty \widehat v_j\varphi_j\in L^2(\Omega):\|v\|_{H(2s)}^2=\sum_{j=0}^\infty\lambda_j^{2s}\vert\widehat v_j\vert^2<\infty\right\}.
\]
The spectral fractional Laplacian of  $u\in H(2s)$ is the function belonging to $L^2(\Omega)$ given by the formula
\begin{equation}\label{As}
\Ds u\ =\ \displaystyle \sum_{j=1}^\infty\lambda_j^s\widehat u_j\,\varphi_j.
\end{equation}
\end{defi}
Note that $C^\infty_c(\Omega)\subset H(2s) \hookrightarrow L^2(\Omega)$. So, the operator $\Ds $ is unbounded, densely defined  and with bounded inverse $\Dms$ in $L^2(\Omega)$. Alternatively, for almost every $x\in\Omega$,
\begin{equation}
\Ds u(x)\ =  PV\int_\Omega[u(x)-u(y)]J(x,y)\;dy+\kappa(x)u(x),\label{As2}
\end{equation}
where, letting $p_\Omega(t,x,y)$ denote the heat kernel of $\left.-\lapl\right\vert_\Omega$, 
\begin{equation}\label{jkappa}
J(x,y)=\frac{s}{\Gamma(1-s)}\int_0^\infty\frac{p_\Omega(t,x,y)}{t^{1+s}}\;dt\qquad\hbox{and}\qquad
\kappa(x)=\frac{s}{\Gamma(1-s)}\int_\Omega\left(1-\int_\Omega p_\Omega(t,x,y)\;dy\right)\frac{dt}{t^{1+s}}
\end{equation}
are respectively the {\it jumping kernel} and the {\it killing measure,}\footnote{in the language of potential theory of killed stochastic processes. 
Note that the integral in \eqref{As2} must be understood in the sense of principal values. 
To see this, look at  \eqref{Jbound}.} see Song, Vondra\v{c}ek \cite[formulas (3.3) and (3.4)]{song-vondra}. 
For the reader's convenience, we provide a proof of \eqref{As2} in the Appendix. 
We assume from now on that  
\[\text{$\Omega$ is of class $C^{1,1}$.}\] 
In particular,  
sharp bounds are known for the heat kernel $p_\Omega$, see \eqref{hkb} below, 
and provide in turn sharp estimates for $J(x,y)$, see \eqref{Jbound} below, 
so that the right-hand side of \eqref{As2} remains well-defined
for every $x\in\Omega$ under the assumption that 
$u\in C^{2s+\eps}_{loc}(\Omega)\cap L^1(\Omega,\delta(x)dx)$ 
for some $\eps>0$. 
This allows us to {\it define} the spectral fractional Laplacian
of functions which {\it do not} vanish on the boundary of $\Omega$. 
As a simple example, observe that the function $u=1$ 
does not belong to $H(2s)$ if $s\ge1/4$, yet it solves
\eqref{main} for $\mu=\kappa$ and $\zeta=0$.

\begin{defi} The Green function and the Poisson kernel of the spectral fractional Laplacian are defined respectively by
\begin{equation}\label{green}
G_\Omega^s(x,y)\ =\ \frac1{\Gamma(s)}\int_0^\infty p_\Omega(t,x,y)\,t^{s-1}\;dt,
\qquad x,\,y\in\Omega,x\neq y,\, s\in(0,1],
\end{equation}
where $p_\Omega$ denotes the heat kernel of $\left.-\lapl\right|_\Omega$, and by
\begin{equation}\label{poisson}
P_\Omega^s(x,y)\ :=\ -\left.\frac{\partial}{\partial\nu_y}\right. G^s_\Omega(x,y), \qquad x\in\Omega, y\in\partial\Omega.
\end{equation}
where $\nu$ is the outward unit normal to $\partial\Omega$.  
\end{defi}

In Section \ref{green-sect}, we shall prove that $P_\Omega^s$ is well-defined
(see Lemma \ref{lemma8}) and review some useful identities
involving the Green function $G_\Omega^s$ 
and the Poisson kernel $P_\Omega^s$. 
Now, let us define weak solutions of \eqref{main}. 
\begin{defi} Consider the test function space
\begin{equation}\label{test}
\T(\Omega)\ :=\ \rest^{-s}C^\infty_c(\Omega)
\end{equation}
and the weight
\begin{equation}\label{h1}
h_1(x) = \int_{\partial\Omega} P^s_\Omega(x,y)\,d\sigma(y),\qquad x\in\Omega.
\end{equation}
Given two Radon measures $\mu\in\mathcal{M}(\Omega)$ 
and $\zeta\in\mathcal{M}(\partial\Omega)$ with
\begin{equation}\label{hypo}
\int_\Omega\delta(x)\;d|\mu|(x)\ <\ \infty,\qquad |\zeta|(\partial\Omega)\ <\ \infty,
\end{equation}
a function $u\in L^1_{loc}(\Omega)$ is a weak solution to
\begin{equation}\label{prob}
\left\{\begin{aligned}
\Ds u &=\ \mu & \hbox{ in }\Omega \\
\frac{u}{h_1} &=\ \zeta & \hbox{ on }\partial\Omega
\end{aligned}\right.
\end{equation}
if, for any $\psi\in\T(\Omega)$,
\begin{equation}\label{byparts}
\int_\Omega u\,\Ds\psi\ =\ \int_\Omega \psi\,d\mu - \int_{\partial\Omega} \frac{\partial\psi}{\partial\nu}\;d\zeta.
\end{equation}
\end{defi}

We shall prove that $\T(\Omega)\subseteq C^1_0(\overline{\Omega})$, 
see Lemma \ref{lemma-test}, 
so that all integrals above are well-defined. 
Equation \eqref{byparts} is indeed a weak formulation of
\eqref{prob}, as the following lemma shows.

\begin{lem}\label{clas-weak} 
\begin{enumerate}
\item (weak solutions are distributional solutions) 
Assume that $u\in L^1_{loc}(\Omega)$ is a weak solution of \eqref{prob}. 
Then in fact, $u\in L^1(\Omega,\delta(x)dx)$ and $\Ds u=\mu$ in the sense
of distributions i.e. for any $\psi\in C^\infty_c(\Omega)$,
$\frac{\Ds\psi}{\delta}$ is bounded and
\[
\int_\Omega u \Ds\psi = \int_\Omega \psi \;d\mu,
\]
moreover the boundary condition holds in the sense that for every $\phi\in C(\super\Omega)$
\[
\lim_{t\downarrow 0}\frac1t\int_{\{\delta(x)<t\}}\frac{u(x)}{h_1(x)}\,\phi(x)\:d\sigma(x)= \int_{\partial\Omega}\phi(x)\:d\zeta(x)
\]
whenever $\mu\in\mathcal{M}(\Omega)$ satisfies \eqref{hypo}
and $\zeta\in L^1(\partial\Omega)$.
\item (for smooth data, weak solutions are classical) Assume that $u\in L^1_{loc}(\Omega)$ is a weak solution of \eqref{prob}, where $\mu\in C^\alpha(\overline\Omega)$ for some $\alpha$ such that $\alpha+2s\not\in\N$ and $\zeta\in C(\partial\Omega)$. Then, $\Ds u$ is well-defined by \eqref{As2} for every $x\in\Omega$, $\Ds u(x)=\mu(x)$ for all $x\in\Omega$ and for all $x_0\in\partial\Omega$,
\[
\lim_{\stackrel{\hbox{\footnotesize $x\to x_0$}}{x\in\Omega}}\frac{u(x)}{h_1(x)} = \zeta(x_0).
\]
\item (classical solutions are weak solutions) Assume that $u\in C^{2s+\eps}_{loc}(\Omega)$ is such that $u/h_1\in C(\super\Omega)$. Let $\mu=\Ds u$ be given by \eqref{As2} and $\zeta=\left. u/h_1\right\vert_{\partial\Omega}$. Then, $u$ is a weak solution of \eqref{prob}.
\end{enumerate}
\end{lem}

We present some facts about harmonic functions in Section \ref{harm-sect} with an eye kept on their singular boundary trace in Section \ref{bb-sect}.

We prove the well-posedness of \eqref{prob} in Section \ref{dir-sect}, namely
\begin{theo}\label{point} Given two Radon measures $\mu\in\mathcal{M}(\Omega)$ and $\zeta\in\mathcal{M}(\partial\Omega)$
such that \eqref{hypo} holds, there exists a unique function $u\in L^1_{loc}(\Omega)$ 
which is a weak solution to \eqref{prob}. 
Moreover, for a.e. $x\in\Omega$,
\begin{equation}\label{repr}
u(x)\ =\ \int_\Omega G^s_\Omega(x,y)\;d\mu(y)+
\int_{\partial\Omega}P_\Omega^s(x,y)\;d\zeta(y)
\end{equation}
and
\begin{equation}\label{cont}
\|u\|_{L^1(\Omega,\delta(x) dx)}\ \leq\ C(\Omega,N,s)\left(\|\delta\mu\|_{\mathcal{M}(\Omega)}+\|\zeta\|_{\mathcal{M}(\partial\Omega)}\right).
\end{equation}
In addition, the following  estimates hold.
\begin{align}
\|u\|_{L^p(\Omega,\delta(x) dx)}\ &\leq\ C_1\|\delta\mu\|_{\mathcal{M}(\Omega)}&\qquad\text{if $\zeta=0$ and $p\in\left[1,\frac{N+1}{N+1-2s}\right)$}\label{lp}\\
\|u\|_{C^\alpha(\overline\omega)}\ &\leq\ C_2\left(\|\mu\|_{L^\infty(\omega)}+\|\zeta\|_{\mathcal M(\partial\Omega)}\right)&\qquad\text{if $\omega\subset\subset\Omega$ and $\alpha\in(0,2s)$}\label{calpha}\\
\|u\|_{C^{2s+\alpha}(\overline\omega)}\ &\leq\ C_3\left(\|\mu\|_{C^\alpha(\overline\omega)}+\|\zeta\|_{\mathcal M(\partial\Omega)}\right)&\qquad\text{if $\omega\subset\subset\Omega$ and $2s+\alpha\not\in\N$.}\label{ctwosplusalpha}
\end{align}
In the above $C_1=C_1(\Omega,N,s,p)$, $C_2=C_2(\Omega,\omega,N,s,\alpha)$, $C_3=C_3(\Omega,\omega,N,s,\alpha)$.
\end{theo}

In Section \ref{nonlin-sect} we solve nonlinear Dirichlet problems, by proving
\begin{theo}\label{nonhom-cor} Let $g(x,t):\Omega\times\R^+\longrightarrow\R^+$ be a Carath\'eodory function and $h:\R^+\to\R^+$ a nondecreasing function such that $g(x,0)=0$ and for a.e. $x\in\Omega$ and all $t>0$,
\[
0\le g(x,t)\le h(t) \qquad\text{where}\qquad  h(\delta^{-(2-2s)})\delta\in L^1(\Omega)
\]
Then, problem
\begin{equation}\label{nonhom-prob}
\left\lbrace\begin{aligned}
\Ds u &=\; -g(x,u) & \hbox{ in }\Omega \\
\frac{u}{h_1} &=\; \zeta & \hbox{ on }\partial\Omega
\end{aligned}\right.
\end{equation}
has a solution $u\in L^1(\Omega,\delta(x)dx)$
for any $\zeta\in C(\partial\Omega), \zeta\ge0$. In addition, if $t\mapsto g(x,t)$ is nondecreasing then the solution is unique.
\end{theo}

Finally, with Section \ref{extra-sect} we prove

\begin{theo}\label{xlarge-theo} Let 
\[
p\in\left(1+s,\frac1{1-s}\right).
\]
Then, there exists a function $u\in L^1(\Omega,\delta(x)dx)\cap C^\infty(\Omega)$ solving
\begin{equation}\label{large-prob}
\left\lbrace\begin{aligned}
\Ds u &=\; -u^p & &\hbox{ in }\Omega, \\
\frac{u}{h_1} &=\; +\infty & & \hbox{ on }\partial\Omega
\end{aligned}\right.
\end{equation}
in the following sense:
the first equality holds pointwise and in the sense of distributions, the boundary condition is understood as a pointwise limit. In addition, there exists a constant $C=C(\Omega, N,s,p)$ such that
\[
0\le u\le C\delta^{-\frac{2s}{p-1}}.
\]
\end{theo}

\begin{rmk}\rm
If one replaces the spectral fractional Laplacian by the classical fractional Laplacian ${(-\lapl)}^s$, then an analogous of Theorem \ref{xlarge-theo} holds for
\[
p\in\left(1+2s,\frac{1+s}{1-s}\right),
\]
see \cite[Theorem 3 and Example 22]{a2}. We suspect that both exponent ranges are optimal.
\end{rmk}

\section{Green function and Poisson kernel}\label{green-sect}
In the following three lemmas\footnote{which hold even if the domain $\Omega$ is not $C^{1,1}$.}, we establish some useful identities for the Green function defined by \eqref{green}.
\begin{lem}\label{inverse}{\rm (\cite[formula (17)]{grss})}
Let $f\in L^2(\Omega)$. For almost every $x\in\Omega$, $G_\Omega^s(x,\cdot)f\in L^1(\Omega)$ and
\[
\rest^{-s}f(x) = \int_\Omega G_\Omega^s(x,y)f(y)\;dy\qquad \text{for a.e. $x\in\Omega$}.
\]
\end{lem}

\begin{proof} If $\varphi_j$ is an eigenfunction of $\left.-\lapl\right\vert_\Omega$, then
\begin{align*}
& \int_\Omega G_\Omega^s(x,y)\varphi_j(y)\; dy \ =\ \int_0^\infty\frac{t^{s-1}}{\Gamma(s)}\int_\Omega p_\Omega(t,x,y)\varphi_j(y)\;dy\;dt \\
& =\ \int_0^\infty\frac{t^{s-1}}{\Gamma(s)}\,e^{-\lambda_jt}\varphi_j(x)\;dt\ =\ \frac{\lambda_j^{-s}}{\Gamma(s)}\varphi_j(x)\int_0^\infty t^{s-1}e^{-t}\;dt
\ =\ \lambda_j^{-s}\varphi_j(x)=\Dms\varphi_j(x)
\end{align*}
By linearity, if $f\in L^2(\Omega)$ is a linear combination of eigenvectors  $f=\sum_{j=1}^M \widehat f_j\varphi_j$, then 
\[
\int_\Omega G_\Omega^s(x,y)\sum_{j=1}^M \widehat f_j\varphi_j(y)\; dy\ =\ \sum_{j=1}^M \widehat f_j\lambda_j^{-s}\varphi_j(x).
\]
and so, letting 
\begin{equation}\label{identity}\mathbb{G}^s_\Omega f\ :=\ \int_\Omega G_\Omega^s(\cdot,y)\,f(y)\;dy,
\end{equation}
we have
\begin{equation}\label{riesz-id}
\|\mathbb{G}^s_\Omega f\|_{H(2s)}^2=\sum_{j=1}^M\lambda_j^{2s}\cdot \vert\widehat f_j\vert^2\lambda_j^{-2s} = \|f\|_{L^2(\Omega)}^2.
\end{equation}
Thus the map $\mathbb{G}_\Omega^s: f\longmapsto \mathbb G^s_\Omega f$
uniquely extends to a linear isometry from $L^2(\Omega)$ to $H(2s)$, 
which coincides with $\Dms$. It remains to prove that the identity \eqref{identity} remains valid a.e. for $f\in L^2(\Omega)$.
By standard parabolic theory, the function $(t,x)\mapsto \int_\Omega p_\Omega(t,x,y)dy$ is bounded (by 1) and smooth in $[0,T]\times\omega$ for every $T>0$, $\omega\subset\subset \Omega$. Hence, for every $x\in\Omega$, $G_\Omega^s(x,\cdot)\in L^1(\Omega)$.
Assume first that $f=\psi\in C^\infty_c(\Omega)$ and take a sequence ${\{\psi_k\}}_{k\in\N}$
in the linear span of the eigenvectors
${\{\varphi_j\}}_{j\in\N}$ such that ${\{\psi_k\}}_{k\in\N}$ 
converges to $\psi$ in $L^2(\Omega)$. The convergence is in fact uniform and so \eqref{identity} holds for $f=\psi$. Indeed,
by standard elliptic regularity, there exist constants $C=C(N,\Omega), k=k(N)$ such that any eigenfunction satisfies 
\[
\|\grad\varphi_j\|_{L^\infty(\Omega)}\leq (C\lambda_j)^k \|\varphi_j\|_{L^2(\Omega)}=(C\lambda_j)^k.
\]
In particular, taking $C$ larger if needed,
\begin{equation}\label{esti-ef}
\left\|\frac{\varphi_j}{\delta}\right\|_{L^\infty(\Omega)}\leq(C\lambda_j)^k.
\end{equation}
Now write $\psi=\sum_{j=1}^{\infty}\widehat\psi_j\varphi_j$ and fix $m\in\N$. Integrating by parts $m$ times yields
\[
\widehat\psi_j=\int_\Omega\psi\varphi_j=-\frac1{\lambda_j}\int_\Omega\psi\,\lapl\varphi_j
=-\frac1{\lambda_j}\int_\Omega\lapl\psi\,\varphi_j=\ldots=\frac{(-1)^m}{\lambda_j^m}\int_\Omega\lapl^m\psi\,\varphi_j
\]
which implies that
\begin{equation}\label{spectral-coef}
|\widehat\psi_j|\leq \frac{\|\lapl^m\psi\|_{L^2(\Omega)}}{\lambda_j^m},
\end{equation}
i.e. the spectral coefficients of $\psi$ converge to $0$ faster than any polynomial. This and \eqref{esti-ef} imply that ${\{\psi_k\}}_{k\in\N}$ 
converges to $\psi$ uniformly, as claimed.

Take at last $f\in L^2$ and a sequence ${\{f_k\}}_{k\in\N}$ 
in $C^\infty_c(\Omega)$ of nonnegative functions such that ${\{f_k\}}_{k\in\N}$ 
converges to $\vert f\vert$ a.e. and in $L^2(\Omega)$. 
By \eqref{riesz-id}, 
$\|\mathbb{G}^s_\Omega f_k\|_{L^2}\le \|f_k\|_{L^2}$ 
and by  Fatou's lemma, we deduce that
$G_\Omega^s(x,\cdot)f\in L^1(\Omega)$ 
for a.e. $x\in\Omega$ and the desired identity follows.
\end{proof}

\begin{lem}{\rm (\cite[formula (8)]{grss})}
For a.e. $x,y\in\Omega$, 
\begin{equation}\label{compo}
\int_{\Omega} G^{1-s}_\Omega(x,\xi)G^s_\Omega(\xi,y)\:d\xi= G^1_\Omega(x,y)
\end{equation}
\end{lem}
\begin{proof}
Clearly, given an eigenfunction $\varphi_j$,
\[
\rest^{-s}\rest^{s-1}\varphi_j = \lambda_j^{-s}\lambda_j^{s-1}\varphi_j = \rest^{-1}\varphi_j
\]
so $\rest^{-s}\circ\rest^{s-1}=\rest^{-1}$ in $L^2(\Omega)$. By the previous lemma and Fubini's theorem, we deduce that for $\varphi\in L^2(\Omega)$ and a.e. $x\in\Omega$,
\[
\int_{\Omega^2} G^{1-s}_\Omega(x,\xi)G^s_\Omega(\xi,y)\varphi(y)\;d\xi\:dy= \int_\Omega G^1_\Omega(x,y)\varphi(y)\; dy
\]
and so \eqref{compo} holds almost everywhere.
\end{proof}

\begin{lem} For any $\psi\in C^\infty_c(\Omega)$, 
\begin{equation}\label{adg}
\Ds\psi\ =\ (-\lapl)\circ \rest^{s-1} \psi
\ =\ \rest^{s-1}\circ(-\lapl)  \psi
\end{equation}
\end{lem}
\begin{proof} 
The identity clearly holds if $\psi$ is an eigenfunction. If $\psi\in C^\infty_c(\Omega)$, its spectral coefficients have fast decay and the result follows by writing the spectral decomposition of $\psi$. Indeed, 
thanks to \eqref{spectral-coef} and \eqref{esti-ef},  we may easily work by density to establish \eqref{adg}.
\end{proof}

Let us turn to the definition and properties of the Poisson kernel. 
Recall that, 
for $x\in\Omega, y\in\partial\Omega$, the Poisson kernel of the Dirichlet Laplacian is given by
\[
P_\Omega^1(x,y)\ =\ -\left.\frac{\partial}{\partial\nu_y}\right. G^1_\Omega(x,y).
\]

\begin{lem}\label{lemma8} 
The function
\[
P_\Omega^s(x,y)\ :=\ -\frac{\partial}{\partial\nu_y} G^s_\Omega(x,y)
\]
is well-defined for $x\in\Omega, y\in\partial\Omega$ and 
$P_\Omega^s(x,\cdot)\in C(\partial\Omega)$ for any $x\in\Omega$.
Furthermore, there exists a constant $C>0$ depending on $N,s,\Omega$ only such that
\begin{equation}\label{poissbound}
\frac1C\frac{\delta(x)}{|x-y|^{N+2-2s}}\le P^s_\Omega(x,y)\le C\frac{\delta(x)}{|x-y|^{N+2-2s}}.
\end{equation}
and 
\begin{equation}\label{pois-id}
\int_{\Omega} G^{1-s}_\Omega(x,\xi)P^s_\Omega(\xi,y)\;d\xi\ =\ P_\Omega^1(x,y).
\end{equation}
\end{lem}
\begin{rmk}\rm
When $\Omega$ is merely Lipschitz, one must work with the Martin kernel in place of the Poisson kernel, see \cite{gprssv}.
\end{rmk}

\begin{proof} {\it of Lemma \ref{lemma8}. }
Take $x,z\in\Omega$, $y\in\partial\Omega$. Then,
\[\frac{G_\Omega^s(x,z)}{\delta(z)}=
\frac{1}{\Gamma(s)\,\delta(z)}\int_0^\infty p_\Omega(t,x,z)t^{s-1}\;dt=
\frac{\vert z-x\vert^{2s}}{\Gamma(s)\,\delta(z)} \int_0^\infty p_\Omega(\vert z-x\vert^{2}\tau,x,z)\tau^{s-1}\;d\tau
\]
Since $\Omega$ has $C^{1,1}$ boundary, 
given $x\in\Omega$, $p_\Omega(\cdot,x,\cdot)\in C^1((0,+\infty)\times\overline\Omega)$ 
and the following heat kernel bound holds (cf. Davies, Simon and Zhang \cite{davies1,davies2, zhang}):
\begin{equation}\label{hkb}
\left[\frac{\delta(x)\delta(y)}{t}\wedge 1\right]\frac{1}{c_1 t^{N/2}}e^{-|x-y|^2/(c_2t)}\le p_\Omega(t,x,y)\leq\left[\frac{\delta(x)\delta(y)}{t}\wedge 1\right]\frac{c_1}{t^{N/2}}e^{-c_2|x-y|^2/t},
\end{equation}
where $c_1,c_2$ are constants depending on $\Omega, N$ only.
So,
\begin{equation}\label{prepois}
\frac{\vert z-x\vert^{2s}}{\,\delta(z)} p_\Omega(\vert z-x\vert^{2}\tau,x,z)\tau^{s-1}\le C
\vert z-x\vert^{2s-N-2}\delta(x)\tau^{s-2-N/2}e^{-c_2/\tau}.
\end{equation}
and the reverse inequality 
\[
\frac 1C
\vert z-x\vert^{2s-N-2}\delta(x)\tau^{s-2-N/2}e^{-1/(c_2\tau)}\le
\frac{\vert z-x\vert^{2s}}{\,\delta(z)} p_\Omega(\vert z-x\vert^{2}\tau,x,z)\tau^{s-1}
\]
also holds for $\tau\ge \delta(x)\delta(z)\vert z-x\vert^{-2}$.
As $z\to y\in\partial\Omega$, 
the right-hand-side of \eqref{prepois} obviously converges 
in $L^1(0,+\infty,d\tau)$ so we may apply 
the generalized dominated convergence theorem to deduce that
$P_\Omega^s(x,y)$ is well-defined, satisfies \eqref{poissbound} and 
\[
P_\Omega^s(x,y)=-\frac{\partial}{\partial\nu_y} G^s_\Omega(x,y)=\lim_{z\to y}\frac{G^s_\Omega(x,z)}{\delta(z)}
=-\frac1{\Gamma(s)}\int_0^\infty\frac{\partial}{\partial\nu_y}p_\Omega(t,x,y)\,t^{s-1}\,dt.
\]
From this last formula we deduce also that, for any fixed $x\in\Omega$,
the function $P_\Omega^s(x,\cdot)\in C(\partial\Omega)$: indeed,
having chosen a sequence ${\{y_k\}}_{k\in\N}\subset\partial\Omega$
converging to some $y\in\partial\Omega$, we have
\[
\left|P_\Omega^s(x,y_k)-P_\Omega^s(x,y)\right|
\leq\frac1{\Gamma(s)}\int_0^\infty\left|\frac{\partial}{\partial\nu_y}p_\Omega(t,x,y_k)
-\frac{\partial}{\partial\nu_y}p_\Omega(t,x,y)\right|t^{s-1}\,dt
\]
where, by\eqref{hkb}
\[
\left|\frac{\partial}{\partial\nu_y}p_\Omega(t,x,y)\right|\leq
\frac{c_1\,\delta(x)}{t^{N/2+1}}\,e^{-c_2|x-y|^2/t}\leq
\frac{c_1\,\delta(x)}{t^{N/2+1}}\,e^{-c_2\delta(x)^2/t}
\qquad\hbox{for any }y\in\partial\Omega,
\]
so that $\left|P_\Omega^s(x,y_k)-P_\Omega^s(x,y)\right|\to 0$ as $k\uparrow\infty$
by dominated convergence.

By similar arguments, $G_\Omega^s$ is a continuous function on $\super\Omega^2\setminus\{(x,y):x= y\}$.
And so, by \eqref{compo}, we have
\[
-\frac{\partial}{\partial\nu_y}\int_{\Omega} G^{1-s}_\Omega(x,\xi)G^s_\Omega(\xi,y)\;d\xi =\ P_\Omega^1(x,y)
\]
Let us compute the derivative of the left-hand side alternatively. We have
\[
\int_\Omega G_\Omega^{1-s}(x,\xi)\frac{G_\Omega^s(\xi,z)}{\delta(z)}\;d\xi= 
 \int_{\R^+\times\Omega} f(t,\xi,z)\;dt d\xi,
\]
where, having fixed $x\in\Omega$,
\[
f(t,\xi,z) = \frac{G_\Omega^{1-s}(x,\xi)}{\Gamma(s)\,\delta(z)}  p_\Omega(t,\xi,z)t^{s-1} \le C \vert x-\xi\vert^{2s-N}\left[\frac{\delta(x)\delta(\xi)}{\vert x-\xi\vert^2}\wedge 1\right]t^{s-2-N/2}\delta(\xi)e^{-c_2\vert z-\xi\vert^2/t}
\]
For fixed $\eps>0$, and $\xi\in\Omega\setminus B(y,\eps)$, $z\in B(y,\eps/2)$, we deduce that
\[
f(t,\xi,z) \le C \vert x-\xi\vert^{2s-N}t^{s-2-N/2}e^{-c_\eps/t}\in L^1((0,+\infty)\times\Omega)
\]
Similarly, if $t>\eps$,
\[
f(t,\xi,z)\le C \vert x-\xi\vert^{2s-N}t^{s-2-N/2}\in L^1((\eps,+\infty)\times\Omega)
\]
Now,
\[
\int_0^\eps t^{s-1-N-2}e^{-c_2\frac{\vert\xi-z\vert^2}{t}}dt \le \vert\xi-z\vert^{2s-N} \int_0^{+\infty} \tau^{s-1-N/2}e^{-c_2/\tau}d\tau
\]
Hence, there exists a constant $C>0$ independent of $\eps$ such that 
\[
\int_{(0,\eps)\times B(y,\eps)}f(t,\xi,z) dt\;d\xi \le C \eps^{2+2s}
\]
It follows from the above estimates and dominated convergence that
\begin{equation}
P_\Omega^1(x,y)=\lim_{z\to y}\int_{(0,+\infty)\times\Omega} f(t,\xi,z)\;dt d\xi = \int_{\Omega} G^{1-s}_\Omega(x,\xi)P^s_\Omega(\xi,y)d\xi
\end{equation}
i.e. \eqref{pois-id} holds.
\end{proof}

\begin{rmk} \rm
Thanks to the heat kernel bound \eqref{hkb}, the following estimate also holds. 
\begin{equation}\label{green-behav}
\frac1C \frac1{{|x-y|}^{N-2s}}\left(1\wedge\frac{\delta(x)\,\delta(y)}{|x-y|^2}\right) \le G_\Omega^s(x,y)\le  \frac C{{|x-y|}^{N-2s}}\left(1\wedge\frac{\delta(x)\,\delta(y)}{|x-y|^2}\right).
\end{equation}
for some constant $C=C(\Omega,N,s)$.
Also observe for computational convenience that 
\[
\frac12\left(1\wedge\frac{\delta(x)\,\delta(y)}{|x-y|^2}\right)\le \frac{\delta(x)\,\delta(y)}{\delta(x)\,\delta(y)+|x-y|^2}\le \left(1\wedge\frac{\delta(x)\,\delta(y)}{|x-y|^2}\right).
\]
\end{rmk}

\section{Harmonic functions and interior regularity}
\label{harm-sect}

\begin{defi}\label{harm-def} 
A function $h\in L^1(\Omega,\delta(x)dx)$ is $s$-harmonic in $\Omega$ 
if for any $\psi\in C^\infty_c(\Omega)$ there holds
\[
\int_\Omega h\:\Ds\psi\ =\ 0.
\]
\end{defi}
\noindent The above definition makes sense thanks to the following lemma.
\begin{lem}\label{As-smooth} 
For any $\psi\in C^\infty_c(\Omega)$, $\Ds\psi\in C^1_0(\super\Omega)$ and there exists a constant $C=C(s,N,\Omega,\psi)>0$ such that
\begin{equation}\label{ineq-lemma1}
|\Ds \psi|\ \leq\ C\delta\qquad \text{in }\Omega.
\end{equation}
In addition, if $\psi\ge0$, $\psi\not\equiv0$, then
\begin{equation}
\Ds\psi\ \leq\ -C\delta \qquad \text{in } \Omega\setminus\hbox{supp}\,\psi.
\end{equation}
\end{lem}
\begin{proof}
Thanks to \eqref{esti-ef} and \eqref{spectral-coef}, 
$\Ds\psi\in C^1_0(\overline\Omega)$ and 
\[
\left|\frac{\Ds \psi}{\delta}\right|\leq
\sum_{j=1}^{\infty}\lambda_j^s|\widehat\psi_j|\left\|\frac{\varphi_j}{\delta}\right\|_{L^\infty(\Omega)}
<\infty
\]
and \eqref{ineq-lemma1} follows. Let us turn to the case where $\psi\ge0$, $\psi\not\equiv0$.
By the heat kernel bound \eqref{hkb}, there exists $C=C(\Omega,N,s)>0$ such that
\begin{equation}\label{Jbound}
\frac{1}{C{|x-y|}^{N+2s}}
\left[\frac{\delta(x)\delta(y)}{{|x-y|}^2}\wedge 1\right]\leq J(x,y)\leq \frac{C}{{|x-y|}^{N+2s}}
\left[\frac{\delta(x)\delta(y)}{{|x-y|}^2}\wedge 1\right].
\end{equation}
Now, we apply formula \eqref{As2} and assume 
that $x\in\Omega\setminus\hbox{supp}\psi$. Denote by $x^*$ a point of maximum of $\psi$ and let $2r=\dist(x^*,\hbox{supp}\,\psi)$. Then for $y\in B_r(x^*)$, it holds $\psi(y),\delta(y)\ge c_1>0$, $\vert x-y\vert\le c_2$ and so
\begin{eqnarray*}
\Ds \psi(x) & = & -\int_\Omega\psi(y)\,J(x,y)\;dy \\
& \leq & -C\int_{B_r(x^*)}\frac{\psi(y)}{{|x-y|}^{N+2s}}
\left[\frac{\delta(x)\delta(y)}{{|x-y|}^2}\wedge 1\right]dy \\
& \leq & -C\,\delta(x).
\end{eqnarray*}
\end{proof}

\begin{lem}\label{pois-harm} 
The function $P_\Omega^s(\cdot,z)\in L^1(\Omega,\delta(x)dx)$ 
is $s$-harmonic in $\Omega$ for any fixed $z\in\partial\Omega$.
\end{lem}
\begin{proof}
Thanks to \eqref{poissbound}, 
$P_\Omega^s(\cdot,z)\in L^1(\Omega,\delta(x)dx)$.
Pick $\psi\in C^\infty_c(\Omega)$ and exploit \eqref{adg}:
\[
\int_\Omega P_\Omega^s(\cdot,z)\,\Ds\psi=
\int_\Omega P_\Omega^s(\cdot,z)\,\rest^{s-1}[-\lapl\psi].
\]
Applying Lemma \ref{inverse}, the Fubini's Theorem and \eqref{pois-id}, the above quantity is equal to
\[
\int_\Omega P^1_\Omega(\cdot,z)\,(-\lapl)\psi=0.
\]
\end{proof}

\begin{lem}\label{harm}  
For any finite Radon measure $\zeta\in\mathcal{M}(\partial\Omega)$, let
\begin{equation}\label{repres}
h(x)\ =\ \int_{\partial\Omega} P^s_\Omega(x,z)\;d\zeta(z),\qquad x\in\Omega.
\end{equation}
Then, 
$h$ is $s$-harmonic in $\Omega$. In addition, there exists a constant $C=C(N,s,\Omega)>0$ such that
\begin{equation}\label{blah}
\Vert h\Vert_{L^1(\Omega,\delta(x)dx)} \le C \Vert \zeta\Vert_{\mathcal M(\partial\Omega)}.
\end{equation}
Conversely, for any $s$-harmonic function $h\in L^1(\Omega,\delta(x)dx)$, $h\geq 0$,
there exists a finite Radon measure $\zeta\in\mathcal{M}(\partial\Omega)$, $\zeta\geq 0$,
such that \eqref{repres} holds.
\end{lem}
\begin{proof} Since $P(x,\cdot)$ is continuous, $h$ is well-defined. By \eqref{poissbound},
\[
\delta(x)|h(x)|\leq
C\int_{\partial\Omega}\frac{d|\zeta|(z)}{|x-z|^{N-2s}}
\]
so that $h\in L^1(\Omega,\delta(x)dx)$ and \eqref{blah} holds.
Pick now $\psi\in C^\infty_c(\Omega)$:
\[
\int_\Omega h(x)\,\Ds\psi(x)\;dx=
\int_{\partial\Omega}\left(\int_\Omega P_\Omega^s(x,z)\,\Ds\psi(x)\;dx\right)d\zeta(z)=0
\]
in view of Lemma \ref{pois-harm}. 
Conversely, let $h$ denote a nonnegative $s$-harmonic function.
By Definition \ref{harm-def} and by equation \eqref{adg}, 
we have for any $\psi\in C^\infty_c(\Omega)$
\begin{multline*}
0=\int_\Omega h(x)\Ds\psi(x)\:dx=
\int_\Omega h(x)\,\rest^{s-1}\circ(-\lapl)\,\psi(x)\:dx\ =\\
=\ \int_\Omega \left(\int_\Omega G_\Omega^{1-s}(x,\xi)h(\xi)\,d\xi\right)(-\lapl\psi(x)\:dx,
\end{multline*}
so that $\int_\Omega G_\Omega^{1-s}(x,\xi)h(\xi)d\xi$ is a (standard) nonnegative harmonic function.
In particular (cf. e.g. \cite[Corollary 6.15]{axler}),
there exists a finite Radon measure $\zeta\in\mathcal{M}(\partial\Omega)$ 
such that $\int_\Omega G_\Omega^{1-s}(x,\xi)h(\xi)d\xi=\int_{\partial\Omega} P_\Omega^1(x,y)d\zeta(y)$. 
We now exploit equation \eqref{pois-id} to deduce that 
\[
\int_\Omega G^{1-s}_\Omega(x,\xi)\left[h(\xi)-\int_{\partial\Omega}P^s_\Omega(\xi,y)d\zeta(y)\right]d\xi =0.
\]
Since 
\[
\int_\Omega \varphi_1(x)\ \left(\int_\Omega G_\Omega^{1-s}(x,\xi) h(\xi)\:d\xi\right)dx=\int_\Omega h\, \rest^{s-1}\varphi_1=\frac1{\lambda_1^{1-s}}\int_\Omega h\,\varphi_1 <\infty,
\]
it holds $\int_\Omega G_\Omega^{1-s}(x,\xi) h(\xi)d\xi\in C^\infty(\Omega)\cap L^1(\Omega,\delta(x)dx)$. Thanks to \eqref{green-behav}, we are allowed to let $G_\Omega^{s}$ act on it. 
By \eqref{compo}, this leads to 
\[
\int_\Omega G^{1}_\Omega(x,\xi)\left[h(\xi)-\int_{\partial\Omega}P^s_\Omega(\xi,y)\;d\zeta(y)\right]d\xi =0.
\]
Take at last $\psi\in C^\infty_c(\Omega)$ and $\varphi= (\left.-\lapl\right\vert_\Omega)^{-1}\psi$. Then,
\[
0=
\int_\Omega \varphi(x)\left[\int_\Omega G^{1}_\Omega(x,\xi)\left[h(\xi)-\int_{\partial\Omega}P^s_\Omega(\xi,y)d\zeta(y)\right]\:d\xi\right]dx=
\int_\Omega \psi(\xi)\left[h(\xi)-\int_{\partial\Omega}P^s_\Omega(\xi,y)\;d\zeta(y)\right]d\xi
\]
and so \eqref{repres} holds a.e. and in fact everywhere thanks to Lemma \ref{clas-harm} below.
\end{proof}

\begin{lem}\label{clas-harm}
Take $\alpha>0$ such that $2s+\alpha\not\in\N$ and $f\in C^{\alpha}_{loc}(\Omega)$. If $u\in L^1(\Omega,\delta(x)dx)$ solves 
\[
\Ds u = f \qquad\text{in $\mathcal D'(\Omega)$,}
\]
then $u\in C^{2s+\alpha}_{loc}(\Omega)$, the above equation holds pointwise, and given any compact sets $K\subset\subset K'\subset\subset\Omega$, there exists a constant $C=C(s,N,\alpha,K,K',\Omega)$ such that
\[
\Vert u \Vert_{C^{2s+\alpha}(K)} \le C \left(\Vert f \Vert_{C^{\alpha}(K')}+\Vert u \Vert_{L^1(\Omega,\delta(x)dx)}\right).
\]
Similarly, if $f\in L^\infty_{loc}(\Omega)$ and $\alpha\in(0,2s)$,
\[
\Vert u \Vert_{C^\alpha(K)} \le C \left(\Vert f \Vert_{L^\infty(K')}+\Vert u \Vert_{L^1(\Omega,\delta(x)dx)}\right).
\]
In particular, if $h$ is $s$-harmonic, then $h\in C^\infty(\Omega)$ and the equality
$\Ds h(x)=0$ holds at every point $x\in\Omega$.
\end{lem}
\begin{proof} We only prove the former inequality, the proof of the latter follows mutatis mutandis.
Given $x\in\Omega$, let
\[
v(x)=\int_\Omega G_\Omega^{1-s}(x,y)u(y)\;dy.
\]
Observe that $v$ is well-defined and 
\begin{equation}\label{esti-l1d}
\Vert v\Vert_{L^1(\Omega,\delta(x)dx)} \le C(\Omega,N,s)\Vert u\Vert_{L^1(\Omega,\delta(x)dx)}
\end{equation}
Indeed, letting $\varphi_1>0$ denote an eigenvector associated to the principal eigenvalue of the Laplace operator, it follows from the Fubini's theorem and Lemma \ref{inverse} that
\[
\int_\Omega \varphi_1(x)\int_\Omega G_\Omega^{1-s}(x,y)\left\vert u(y)\right\vert dy\, dx = \lambda_1^{s-1}\int_\Omega \left\vert u(y)\right\vert\varphi_1(y)\, dy. 
\]
In addition, $-\lapl v= f$ in $\mathcal D'(\Omega)$, since for $\varphi\in C^\infty_c(\Omega)$,
\[
\int_\Omega v\:(-\lapl)\varphi = 
\int_{\Omega}u\:\Ds\varphi,
\]
thanks to the Fubini's theorem, equation \eqref{adg}, Lemma \ref{inverse} and Definition \ref{harm-def}. 
Observe now that if $\varphi\in C^\infty_c(\Omega)$, then 
\[
\Dums\varphi=\frac{s}{\Gamma(1-s)}\int_0^{+\infty}t^{s-1}
\left(\frac{\varphi-e^{\trest}\varphi}{t}\right)dt.
\]
The above identity is straightforward if $\varphi$ is an eigenfunction and remains true for $\varphi\in C^\infty_c(\Omega)$ by density, using the fast decay of spectral coefficients, see \eqref{spectral-coef}. 
So,
\begin{multline*}
\int_\Omega u\varphi\:dx= 
\int_\Omega v\Dums\varphi\:dx\ =\\
=\ \frac{s}{\Gamma(1-s)}\int_{\Omega}\int_0^\infty
v\,t^{s-1}\left(\frac{\varphi-e^{\trest}\varphi}{t}\right)dt\:dx\ =\\
=\ \frac{s}{\Gamma(1-s)}\int_{\Omega}\int_0^\infty 
\varphi t^{s-1}\left(\frac{v-e^{\trest}v}{t}\right)dt\:dx
\end{multline*}
and 
\[
u = \frac{s}{\Gamma(1-s)}\int_0^{+\infty}t^{s-1}
\left(\frac{v-e^{\trest}v}{t}\right)dt.
\]
Choose $\super f\in C^\alpha_c(\R^N)$ such that $\super f=f$ in $K'$, $\Vert \super f\Vert_{C^{\alpha}(\R^N)}\le C \Vert f\Vert_{C^{\alpha}(K')}$ and let
\[
\super u(x) = c_{N,s}\int_{\R^N}\vert x-y\vert^{-(N-2s)}\super f(y)\;dy
\]
solve $(-\lapl)^s \super u = \super f$ in $\R^N$. It is well-known (see e.g. \cite{silvestre}) that $\Vert \super u\Vert_{C^{2s+\alpha}(\R^N)}\le C \Vert \super f\Vert_{C^{\alpha}(\R^N)}$, for a constant $C$ depending only on $s,\alpha,N$ and the measure of the support of $\super f$. It remains to estimate $u-\super u$. Letting 
\[
\super v(x) = c_{N,1-s}\int_{\R^N} \vert x-y\vert^{-(N-2(1-s))}\super u(y)\;dy,
\]
we have as previously that $-\lapl\super v=\super f$ and
\[
\super u = \frac{s}{\Gamma(1-s)}\int_0^{+\infty}t^{s-1}
\left(\frac{\super v-e^{-t\lapl}\super v}{t}\right)dt,
\]
where this time $e^{-t\lapl}\super v(x)= \frac1{(4\pi t)^{N/2}}\int_{\R^N}e^{-\frac{\vert x-y\vert^2}{4t}}\super v(y)\;dy$.
Hence,
\[
\frac{\Gamma(1-s)}s(u-\super u) = \int_0^{+\infty}t^{s-1}
\left(\frac{(v-\super v)-e^{-t\lapl}(v-\super v)}{t}\right)dt + \int_0^{+\infty}t^{s-1}
\left(\frac{e^{\trest}v-e^{-t\lapl}\super v}{t}\right)dt.
\]
Fix a compact set $K''$ such that $K\subset\subset K''\subset\subset K'$.
Since $v-\super v$ is harmonic in $K'$,
\[\Vert v-\super v\Vert_{C^{2s+\alpha+2}(K'')}\le C\Vert v-\super v\Vert_{L^1(K')}\le C\Vert u\Vert_{L^1(\Omega,\delta(x)dx)}
\]
By parabolic regularity,
\[
\left\Vert\int_0^{+\infty}t^{s-1}
\left(\frac{(v-\super v)-e^{-t\lapl}(v-\super v)}{t}\right)dt
\right\Vert_{C^{2s+\alpha}(K)}\le C\Vert u\Vert_{L^1(\Omega,\delta(x)dx)}.
\]
In addition, the function $w=e^{\trest}v-e^{-t\lapl}\super v$ solves the heat equation inside $\Omega$ with initial condition $w(0,\cdot)=v-\super v$. Since $\lapl (v-\super v)=0$ in $K'$, it follows from parabolic regularity again that $w(t,x)/t$ remains bounded in $C^{2s+\alpha}(K)$ as $t\to0^+$, so that again
\[
\left\Vert\int_0^{+\infty}t^{s-1}
\left(\frac{e^{\trest}v-e^{-t\lapl}\super v}{t}\right)dt\right\Vert_{C^{2s+\alpha}(K)}\le C\Vert u\Vert_{L^1(\Omega,\delta(x)dx)}.
\]

\end{proof}

\begin{lem}\label{u-to-lapl} Take $\alpha>0$, $\alpha\not\in\N$,
and 
$u\in C^{2s+\alpha}_{loc}(\Omega)\cap L^1(\Omega,\delta(x)dx)$.
Given any compact set $K\subset\subset K'\subset\subset\Omega$, there exists a constant $C=C(s,N,\alpha,K, K',\Omega)$ such that
\[
\|\Ds u\|_{C^\alpha(K)}\leq
C\left(\|u\|_{C^{2s+\alpha}(K')} +
\|u\|_{L^1(\Omega,\delta(x)dx)}\right).
\]
\end{lem}
\begin{proof} With a slight abuse of notation, we write
\[
\Dsmu u(x) = \int_\Omega G_\Omega^{1-s}(x,y)\,u(y)\:dy.
\]
By Lemma \ref{clas-harm} we have
\begin{multline*}
\|\Dsmu u\|_{C^{2+\alpha}(K)}\leq C\left(
\|u\|_{C^{2s+\alpha}(K')}+\|\Dsmu u\|_{L^1(\Omega,\delta(x)dx)}\right)\ \leq \\
\leq\ C\left(
\|u\|_{C^{2s+\alpha}(K')}+\|u\|_{L^1(\Omega,\delta(x)dx)}\right).
\end{multline*}
Obviously it holds also
\[
\|(-\lapl)\circ\Dsmu u\|_{C^\alpha(K)}\leq\|\rest^{s-1}u\|_{C^{2+\alpha}(K)}.
\]
By \eqref{adg},
\[
(-\lapl)\circ\Dsmu u
=\Ds u \qquad \hbox{in }\mathcal{D}'(\Omega),
\]
which concludes the proof.
\end{proof}

\begin{prop}\label{Lp-reg} Let $f\in L^1(\Omega,\delta(x)dx)$ and $u\in L^1_{loc}(\Omega)$. 
The function
\[
u(x)=\int_\Omega G_\Omega^s(x,y)\,f(y)\;dy
\]
belongs to $L^p(\Omega,\delta(x)dx)$ for any $p\in\left[1,\dfrac{N+1}{N+1-2s}\right)$.
\end{prop}
\begin{proof}
We start by applying the Jensen's Inequality, 
\[
\left|\int_\Omega G_\Omega^s(x,y)\,f(y)\;dy\right|^p
\leq \|f\|^{p-1}_{L^1(\Omega,\delta(x)dx)}
\int_\Omega\left|\frac{G_\Omega^s(x,y)}{\delta(y)}\right|^p\delta(y)\,f(y)\;dy,
\]
so that
\[
\int_\Omega|u(x)|^p\,\delta(x)\;dx\leq \|f\|^p_{L^1(\Omega,\delta(x)dx)}
\sup_{y\in\Omega}\int_\Omega\left|\frac{G_\Omega^s(x,y)}{\delta(y)}\right|^p\delta(x)\;dx
\]
and by \eqref{green-behav} we have to estimate
\[
\sup_{y\in\Omega}\int_\Omega\frac1{\left|x-y\right|^{(N-2s)p}}\cdot
\frac{\delta(x)^{p+1}}{\left[|x-y|^2+\delta(x)\delta(y)\right]^p}\;dx
\]
Pick $\eps>0$. Clearly,
\[
\sup_{\{y\;:\;\delta(y)\ge\eps\}}\int_\Omega\frac1{\left|x-y\right|^{(N-2s)p}}\cdot
\frac{\delta(x)^{p+1}}{\left[|x-y|^2+\delta(x)\delta(y)\right]^p}\;dx\le C_\eps.
\]
Thanks to Lemma \ref{lemma39}, we may now reduce to the case where the boundary is flat, i.e. when in a neighbourhood $A$ of a given point
$y\in\Omega$ such that $\delta(y)<\eps$, there holds $A\cap\partial\Omega\subseteq\{y_N=0\}$ and $A\cap\Omega\subseteq\{y_N>0\}$. Without loss of generality, we assume that $y=(0,y_N)$ and 
$x=(x',x_N)\in B\times(0,1)\subseteq\R^{N-1}\times\R$.
We are left with proving that
\[
\int_B dx'\int_0^1 dx_N \frac1{\left[|x'|^2+|x_N-y_N|^2\right]^{(N-2s)p/2}}\cdot
\frac{x_N^{p+1}}{\left[|x'|^2+|x_N-y_N|^2+x_Ny_N\right]^p}
\]
is a bounded quantity. Make the change of variables $x_N=y_N t$ and pass to
polar coordinates in $x'$, with $|x'|=y_N\rho$. Then, the above integral becomes
\begin{equation}\label{4576}
y_N^{-(N+1-2s)p+N+1}\int_0^{1/y_N}d\rho\int_0^{1/y_N}dt\;\frac{\rho^{N-2}}{\left[\rho^2+|t-1|^2\right]^{(N-2s)p/2}}
\cdot\frac{t^{p+1}}{\left[\rho^2+|t-1|^2+t\right]^p}.
\end{equation}
Now, we split the integral in the $t$ variable into $\int_0^{1/2}+\int_{1/2}^{3/2}+\int_{3/2}^{1/y_N}$.
Note that the exponent $-(N+1-2s)p+N+1$ is positive for $p<(N+1)/(N+1-2s)$.
We drop multiplicative constants in the computations that follow.
The first integral is bounded above by a constant multiple of
\[
\int_0^{1/y_N}d\rho\int_0^{1/2}dt\;\frac{\rho^{N-2}}{\left[\rho^2+1\right]^{(N-2s)p/2}}
\cdot\frac{t^{p+1}}{\left[\rho^2+1+t\right]^p}\lesssim
\int_0^{1/y_N}d\rho\;\frac{\rho^{N-2}}{\left[\rho^2+1\right]^{(N+2-2s)p/2}}
\]
which remains bounded as $y_N\downarrow 0$ since
\[
p\geq 1>\frac{N-1}{N+2-2s}\qquad\hbox{implies}\qquad (N+2-2s)p-N+2>1.
\]
The second integral is of the order of
\begin{align*}
& \int_0^{1/y_N}d\rho\int_{1/2}^{3/2}dt\;\frac{\rho^{N-2}}{\left[\rho^2+|t-1|^2\right]^{(N-2s)p/2}}\cdot\frac1{\left[\rho^2+1\right]^p} \ =\\
& =\ \int_0^{1/y_N}d\rho\int_0^{1/2}dt\;\frac{\rho^{N-2}}{\left[\rho^2+t^2\right]^{(N-2s)p/2}}\cdot\frac1{\left[\rho^2+1\right]^p} \\
& =\ \int_0^{1/y_N}d\rho\;\frac{\rho^{N-1-(N-2s)p}}{\left[\rho^2+1\right]^p}\int_0^{1/(2\rho)}dt\;\frac1{\left[1+t^2\right]^{(N-2s)p/2}} \\
& \lesssim\ \int_0^\infty d\rho\;\frac{\rho^{N-1-(N-2s)p}}{\left[\rho^2+1\right]^p}
\end{align*}
which is finite since $p<(N+1)/(N+1-2s)<N/(N-2s)$ implies $N-1-(N-2s)p>-1$ 
and $p\geq 1>N/(N+2-2s)$ implies $2p-N+1+(N-2s)p>1$.

We are left with the third integral which is controlled by
\begin{align*}
& \int_0^{1/y_N}d\rho\;\rho^{N-2}\int_{3/2}^{1/y_N}dt\;\frac{t^{p+1}}{\left[\rho^2+t^2\right]^{(N+2-2s)p/2}}\ \leq\\
& \leq\ \int_0^{1/y_N}d\rho\;\rho^{N-(N+1-2s)p}\int_{3/(2\rho)}^{1/(y_N\rho)}dt\;\frac{t^{p+1}}{\left[1+t^2\right]^{(N+2-2s)p/2}}\\
& \lesssim\ \int_0^{1/y_N}d\rho\;\rho^{N-(N+1-2s)p}.
\end{align*}
The exponent $N-(N+1-2s)p>-1$ since $p<(N+1)/(N+1-2s)$, so this third integral is bounded above by a constant multiple of
$y_N^{-1-N+(N+1-2s)p}$ which simplifies with the factor in front of \eqref{4576}.
\end{proof}

\section{Boundary behaviour}\label{bb-sect}

We first provide the boundary behaviour of the reference function $h_1$. Afterwards, in Proposition \ref{bound-cont} below, we will deal with the weighted trace left on the boundary by harmonic functions induced by continuous boundary data.

\begin{lem} Let $h_1$ be given by \eqref{h1}. There exists a constant $C=C(N,\Omega,s)>0$ such that
\begin{equation}\label{h1-behav}
\frac1C \delta^{-(2-2s)}\le h_1\le C \delta^{-(2-2s)}.
\end{equation}
\end{lem}
\begin{proof}
Restrict without loss of generality to the case where $x$ lies in a neighbourhood of $\partial\Omega$.
Take $x^*\in\partial\Omega$ such that $|x-x^*|=\delta(x)$, 
which exists by compactness of $\partial\Omega$.
Take $\Gamma\subset\partial\Omega$ a neighbourhood of $x^*$ in the topology of $\partial\Omega$. By Lemma \ref{lemma39} in the Appendix,
we can think of $\Gamma\subset\{x_N=0\}$, $x^*=0$ and $x=(0,\delta(x))\in\R^{N-1}\times\R$ without loss of generality.
in such a way that it is possible to compute
\[
\int_\Gamma \frac{\delta(x)}{|x-z|^{N+2-2s}}\;d\sigma(z)
\asymp\int_{B_r}\frac{\delta(x)}{[|z'|^2+\delta(x)^2]^{N/2+1-s}}\;dz'.
\]
Recalling \eqref{poissbound}, we have reduced the estimate to 
\begin{align*}
& \int_{\partial\Omega}P^s_\Omega(x,z)\;d\sigma(z)\ \asymp \\
& \asymp\ \int_{B_r}\frac{\delta(x)}{[|z'|^2+\delta(x)^2]^{N/2+1-s}}\;dz'
\ \asymp\ \int_0^r\frac{\delta(x)\,t^{N-2}}{[t^2+\delta(x)^2]^{N/2+1-s}}\;dt \\
& =\ \int_0^{r/\delta(x)}\frac{\delta(x)^{N}t^{N-2}}{[\delta(x)^2t^2+\delta(x)^2]^{N/2+1-s}}\;dt\ \asymp\ \delta(x)^{2s-2}\int_0^{r/\delta(x)}\frac{t^{N-2}\;dt}{[t^2+1]^{N/2+1-s}}
\end{align*}
and this concludes the proof, since
\[
\int_0^{r/\delta(x)}\frac{t^{N-2}\;dt}{[t^2+1]^{N/2+1-s}}\asymp 1.
\]
\end{proof}

In the following we will use the notation
\[
\mathbb P_\Omega^sg\ :=\ \int_{\partial\Omega}P_\Omega^s(\cdot,\theta)\,g(\theta)\;d\sigma(\theta)
\]
where $\sigma$ denotes the Hausdorff measure on $\partial\Omega$,
whenever $g\in L^1(\Omega)$.

\begin{prop}\label{bound-cont} Let $\zeta\in C(\partial\Omega)$. Then, for any $z\in\partial\Omega$,
\begin{equation}
\frac{\mathbb P_\Omega^s \zeta(x)}{h_1(x)}\ \xrightarrow[x\to z]{x\in\Omega} \ \zeta(z)\qquad\hbox{uniformly on }\partial\Omega.
\end{equation}
\end{prop}
\begin{proof}
Let us write
\begin{multline*}
\left|\frac{\mathbb P_\Omega^s\zeta(x)}{h_1(x)}-\zeta(z)\right|=
\left|\frac1{h_1(x)}\int_{\partial\Omega}P^s_\Omega(x,\theta)\,\zeta(\theta)\;d\sigma(\theta)-\frac{h_1(x)\,\zeta(z)}{h_1(x)}\right|\ \leq\\
\leq\ \frac1{h_1(x)}\int_{\partial\Omega}P^s_\Omega(x,\theta)|\zeta(\theta)-\zeta(z)|\;d\sigma(\theta)\leq
C\delta(x)^{3-2s}\int_{\partial\Omega}\frac{|\zeta(\theta)-\zeta(z)|}{|x-\theta|^{N+2-2s}}\;d\sigma(\theta)\ \leq \\
\leq\ C\delta(x)\int_{\partial\Omega}\frac{|\zeta(\theta)-\zeta(z)|}{|x-\theta|^N}\;d\sigma(\theta).
\end{multline*}
It suffices now to repeat the computations in 
\cite[Lemma 3.1.5]{a1} 
to show that the obtained quantity 
converges to $0$ as $x\to z$.
\end{proof}

With an approximation argument started from the last Proposition,
we can deal with a $\zeta\in L^1(\partial\Omega)$ datum.

\begin{theo}\label{bound-l1} For any $\zeta\in L^1(\partial\Omega)$
and any $\phi\in C^0(\super\Omega)$ it holds
\[
\frac1t\int_{\{\delta(x)\leq t\}}
\frac{\mathbb{P}_\Omega^s\zeta(x)}{h_1(x)}\,\phi(x)\;dx
\xrightarrow[t\downarrow 0]{} \int_{\partial\Omega}\phi(y)\,\zeta(y)\;d\sigma(y).
\]
\end{theo}
\begin{proof}
For a general $\zeta\in L^1(\partial\Omega)$, consider a sequence
${\{\zeta_k\}}_{k\in\N}\subset C(\partial\Omega)$ such that 
\begin{equation}\label{topos}
\int_{\partial\Omega}\left|\zeta_k(y)-\zeta(y)\right|d\sigma(y)\xrightarrow[k\uparrow\infty]{} 0.
\end{equation}
For any fixed $k\in\N$, we have
\begin{align}
& \left| \frac1t\int_{\{\delta(x)<t\}}\frac{\mathbb P_\Omega^s\zeta(x)}{h_1(x)}\,\phi(x)\:dx 
- \int_{\partial\Omega}\phi(x)\,\zeta(x)\:d\sigma(x) \right| \leq \nonumber \\
& \left|\frac1t\int_{\{\delta(x)<t\}}\frac{\mathbb P_\Omega^s\zeta(x)-\mathbb P_\Omega^s\zeta_k(x)}{h_1(x)}\,\phi(x)\:dx \right| \label{1first}\\
& +\ \left| \frac1t\int_{\{\delta(x)<t\}}\frac{\mathbb P_\Omega^s\zeta_k(x)}{h_1(x)}\,\phi(x)\:dx 
- \int_{\partial\Omega}\phi(x)\,\zeta_k(x)\:d\sigma(x) \right| \label{2second} \\
& +\ \left|\int_{\partial\Omega}\phi(x)\,\zeta_k(x)\:d\sigma(x)-\int_{\partial\Omega}\phi(x)\,\zeta(x)\:d\sigma(x) \right|. \label{3third}
\end{align}
Call $\lambda_k:=\zeta_k-\zeta$: 
the term \eqref{1first} equals
\[
\frac1t\int_{\{\delta(x)<t\}}\frac{\mathbb P_\Omega^s\lambda_k(x)}{h_1(x)}\,\phi(x)\:d\sigma(x) =
\int_{\partial\Omega}\left(\frac1t\int_{\{\delta(x)<t\}}\frac{P_\Omega^s(x,y)}{h_1(x)}\,\phi(x)\:dx\right)\lambda_k(y)\:d\sigma(y)
\]
Call
\[
\Phi(t,y):=\frac1t\int_{\{\delta(x)<t\}}\frac{P_\Omega^s(x,y)}{h_1(x)}\,\phi(x)\:dx.
\]
Combining equations \eqref{poissbound}, \eqref{h1-behav} and the boundedness of $\phi$,
we can prove that $\Phi$ is uniformly bounded in $t$ and $y$. Indeed,
\[
|\Phi(t,y)|\leq\frac{\|\phi\|_{L^\infty(\Omega)}}{t}
\int_{\{\delta(x)<t\}}\frac{\delta(x)^{3-2s}}{|x-y|^{N+2-2s}}\:dx
\leq
\frac{\|\phi\|_{L^\infty(\Omega)}}{t}
\int_{\{\delta(x)<t\}}\frac{\delta(x)}{|x-y|^N}\:dx
\]
and reducing our attention to the flat case (see Lemma \ref{lemma39} in the Appendix for the complete justification)
we estimate (the ' superscript denotes an object living in $\R^{N-1}$)
\[
\frac1t\int_0^t\int_{B'} \frac{x_N}{\left[|x'|^2+x_N^2\right]^{N/2}}\:dx'\:dx_N
\leq \frac1t\int_0^t\int_{B'_{1/x_N}}\frac{d\xi}{\left[|\xi|^2+1\right]^{N/2}}\:dx_N
\leq \int_{\R^{N-1}}\frac{d\xi}{\left[|\xi|^2+1\right]^{N/2}}.
\]
Thus $\int_{\partial\Omega}\Phi(t,y)\lambda_k(y)d\sigma(y)$ 
is arbitrarily small in $k$ in view of \eqref{topos}.

The term \eqref{2second} converges to $0$ as $t\downarrow0$ because the convergence
\[
\frac{\mathbb P_\Omega^s\zeta_k(x)}{h_1(x)}\,\phi(x)\xrightarrow[x\to z]{x\in\Omega} \zeta_k(z)\,\phi(z)
\]
is uniform in $z\in\partial\Omega$ in view of Proposition \ref{bound-cont}.

Finally, the term \eqref{3third} is arbitrarily small with $k\uparrow+\infty$, because of \eqref{topos}.
This concludes the proof of the theorem, because 
\begin{multline*}
\lim_{t\downarrow 0}\left| \frac1t\int_{\{\delta(x)<t\}}\frac{\mathbb P_\Omega^s\zeta(x)}{h_1(x)}\,\phi(x)\:dx 
- \int_{\partial\Omega}\phi(x)\,\zeta(x)\:d\sigma(x) \right|
\ \leq \\
\leq\ \|\Phi\|_{L^\infty((0,t_0)\times\partial\Omega)}\int_{\partial\Omega}|\zeta_k(y)-\zeta(y)|\;d\sigma(y)
+\|\phi\|_{L^\infty(\partial\Omega)}\int_{\partial\Omega}|\zeta_k(y)-\zeta(y)|\;d\sigma(y)
\end{multline*}
and letting $k\uparrow+\infty$ we deduce the thesis as a consequence of \ref{topos}.
\end{proof}

Moreover we have also
\begin{theo}\label{bound-g} 
For any $\mu\in\mathcal{M}(\Omega)$, such that 
\begin{equation}\label{mu1}
\int_\Omega\delta\,d|\mu|<\infty,
\end{equation}
and any $\phi\in C^0(\super\Omega)$ it holds
\begin{equation}\label{333}
\frac1t\int_{\{\delta(x)\leq t\}}\frac{\mathbb{G}_\Omega^s\mu(x)}{h_1(x)}\,\phi(x)
\:dx
\ \xrightarrow[t\downarrow 0]{}\ 0.
\end{equation}
\end{theo}
\begin{proof}
By using the Jordan decomposition of $\mu=\mu^+-\mu^-$ into its positive and negative part, we can suppose without loss of generality that $\mu\geq 0$. 
Fix some $s'\in(0,s\wedge 1/2)$. Exchanging the order of integration we claim that
\begin{equation}\label{cla}
\int_{\{\delta(x)\leq t\}}G_\Omega^s(x,y)\,\delta(x)^{2-2s}dx\leq
\left\lbrace\begin{aligned}
& C\,t^{2-2s'}\,\delta(y)^{2s'} & \delta(y)\geq t \\
& C\,t\,\delta(y) & \delta(y)<t
\end{aligned}\right.
\end{equation}
where $C=C(N,\Omega,s)$ and does not depend on $t$, which yields
\[
\frac1t\int_\Omega\left(\int_{\{\delta(x)<t\}}G_\Omega^s(x,y)\,\frac{dx}{h_1(x)}\right)d\mu(y)
\ \leq\ C\,t^{1-2s'}\int_{\{\delta(y)\geq t\}}\delta(y)^{2s'}d\mu(y)
+C\int_{\{\delta(x)<t\}}\delta(y)\:d\mu(y).
\]
The second addend converges to $0$ as $t\downarrow0$ by \eqref{mu1}.
Since $t^{1-2s'}\delta(y)^{2s'}$ converges pointwisely to $0$ in $\Omega$ as $t\downarrow0$ and
$t^{1-2s'}\delta(y)^{2s'}\leq\delta(y)$ in $\{\delta(y)\geq t\}$, 
then the first addend converges to $0$ by dominated convergence.
This suffices to deduce our thesis \eqref{333}.

Let us turn now to the proof of the claimed estimate \eqref{cla}. 
For the first part we refer to \cite[Proposition 7]{dhifli} to say
\[
\int_{\{\delta(x)\leq t\}}G_\Omega^s(x,y)\,\delta(x)^{2-2s}dx\leq
t^{2-2s'}\int_{\{\delta(x)\leq t\}}G_\Omega^s(x,y)\,\delta(x)^{-2s+2s'}dx\leq C\,t^{2-2s'}\,\delta^{2s'}.
\]
We focus our attention on the case where $\partial\Omega$ is locally flat, 
i.e. we suppose that in a neighbourhood $A$ of $y$ it holds $A\cap\partial\Omega\subseteq\{x_N=0\}$
(see Lemma \ref{lemma39} in the Appendix to reduce the general case to this one).
So, since $\delta(x)=x_N$ and retrieving estimate \eqref{green-behav} on the Green function, 
we are dealing with
(the $'$ superscript denotes objects that live in $\R^{N-1}$)
\[
\int_0^t\int_{B'(y')}\frac{y_N\,x_N^{3-2s}}{|x'-y'|^2+(x_N-y_N)^2+x_N\,y_N}
\cdot\frac{dx'}{\left[|x'-y'|^2+(x_N-y_N)^2\right]^{(N-2s)/2}}\;dx_N.
\]
From now on we drop multiplicative constants depending only on $N$ and $s$.
Suppose without loss of generality $y'=0$. Set $x_N=y_N\eta$ and switch to polar coordinates in the $x'$ variable:
\[
\int_0^{t/y_N}\int_0^1\frac{y_N^{5-2s}\,\eta^{3-2s}}{r^2+y_N^2(\eta-1)^2+\eta\,y_N^2}\cdot\frac{r^{N-2}dr}{\left[r^2+y_N^2(\eta-1)^2\right]^{(N-2s)/2}}\;d\eta
\]
and then set $r=y_N\rho$ to get
\begin{multline*}
y_N^2\int_0^{t/y_N}\eta^{3-2s}\int_0^{1/y_N}\frac{\rho^{N-2}}{\left[\rho^2+(\eta-1)^2\right]^{(N-2s)/2}}
\cdot\frac{d\rho}{\rho^2+(\eta-1)^2+\eta}\;d\eta\ \leq\\
\leq\ y_N^2\int_0^{t/y_N}\eta^{3-2s}\int_0^{1/y_N}\frac{\rho}{\left[\rho^2+(\eta-1)^2\right]^{(3-2s)/2}}
\cdot\frac{d\rho}{\rho^2+(\eta-1)^2+\eta}\;d\eta.
\end{multline*}
Consider now $s\in(1/2,1)$. 
The integral in the $\rho$ variable is less than
\[
\int_0^{1/y_N}\frac{\rho}{\left[\rho^2+(\eta-1)^2\right]^{(5-2s)/2}}\;d\rho\leq|\eta-1|^{-3+2s}
\]
so that, integrating in the $\eta$ variable,
\[
y_N^2\int_0^{t/y_N}\eta^{3-2s}\,|\eta-1|^{-3+2s}\:d\eta
\leq t\,y_N
\]
and we prove \eqref{cla} in the case $s\in(1/2,1)$.
Now we study the case $s\in(0,1/2]$. Split the integration in the $\eta$ variable into $\int_0^2$ and $\int_2^{1/y_N}$: the latter can be treated in the same way as above. For the other one we exploit the inequality $\rho^2+(\eta-1)^2+\eta\geq \frac34$ to deduce\footnote{In the computation that follows, in the particular case $s=\frac12$ the term $|1-\eta|^{2s-1}$ must be replaced by $-\ln|1-\eta|$, but this is harmless.}:
\begin{multline*}
y_N^2\int_0^2\eta^{3-2s}\left(\int_0^{1/y_N}\frac{\rho}{\left[\rho^2+(\eta-1)^2+\eta\right]^{(3-2s)/2}}\cdot\frac{d\rho}{\rho^2+(\eta-1)^2}\right)d\eta\ \leq\\
\leq\ y_N^2\int_0^2\eta^{3-2s}\left(\int_0^{1/y_N}\frac{\rho}{\left[\rho^2+(\eta-1)^2\right]^{(3-2s)/2}}\:d\rho\right)d\eta\leq y_N^2\int_0^2\frac{\eta^{3-2s}}{|\eta-1|^{1-2s}}\:d\eta
\leq y_N^2.
\end{multline*}
Note now that, in our set of assumptions, $y_N=\delta(y)<t$. So $y_N^2\leq ty_N$ and we get to the desired conclusion \eqref{cla} also in the case $s\in(0,1/2]$.
\end{proof}

\section{The Dirichlet problem}\label{dir-sect}
Recall the definition of test functions \eqref{test}.

\begin{lem}\label{lemma-test} $\T(\Omega)\subseteq C^1_0(\overline{\Omega})\cap C^\infty(\Omega)$. 
Moreover, for any $\psi\in\T(\Omega)$ and $z\in\partial\Omega$,
\begin{equation}\label{dertest}
-\frac{\partial\psi}{\partial\nu}(z)\ =\ \int_\Omega P_\Omega^s(y,z)\,\Ds\psi(y)\;dy.
\end{equation}
\end{lem}
\begin{proof}
Take $\psi\in\T(\Omega)$ and let $f=\Ds\psi$. Since $f\in C^\infty_c(\Omega)$, the spectral coefficients of $f$ have fast decay (see \eqref{spectral-coef}) and so the same holds true for $\psi$. It follows that $\psi\in C^1_0(\overline\Omega)$ and $\T(\Omega)\subseteq C^1_0(\overline{\Omega})$.
By Lemma \ref{inverse}, for all $x\in\overline\Omega$,
\[
\psi(x) = \int_\Omega G_\Omega^s(x,y)\,f(y)\:dy
\]
Using Lemma \ref{lemma8}, \eqref{prepois} and the dominated convergence theorem, \eqref{dertest} follows.

Since $\Ds$ is self-adjoint in $H(2s)$, we know that the equality $\Ds\psi=f$ holds in $\mathcal{D}'(\Omega)$ and the interior regularity follows from Lemma \ref{clas-harm}.
\end{proof}

\begin{lem}[Maximum principle for classical solutions]\label{maxprinc-clas}
Let $u\in C^{2s+\eps}_{loc}(\Omega)\cap L^1(\Omega,\delta(x)\,dx)$ such that
\[
\Ds u \geq 0\ \hbox{ in }\Omega,\qquad \liminf_{x\to\partial\Omega}u(x)\geq 0.
\]
Then $u\geq 0$ in $\Omega$. In particular this holds when $u\in\T(\Omega)$.
\end{lem}
\begin{proof} Suppose $x^*\in\Omega$ such that 
$\displaystyle u(x^*)=\min_\Omega u<0$. Then
\[
\Ds u(x^*)=\int_\Omega [u(x^*)-u(y)]\,J(x,y)\;dy + \kappa(x^*) u(x^*) < 0,
\]
a contradiction.
\end{proof}

\begin{lem}[Maximum principle for weak solutions] \label{maxprinc-weak}
Let $\mu\in\mathcal{M}(\Omega),\zeta\in\mathcal{M}(\partial\Omega)$ be two Radon measures satisfying \eqref{hypo}
with $\mu\geq 0$ and $\zeta\geq 0$.
Consider $u\in L^1_{loc}(\Omega)$ 
a weak solution to the Dirichlet problem \eqref{prob}. Then $u\geq 0$ a.e. in $\Omega$. 
\end{lem}
\begin{proof}
Take $f\in C^\infty_c(\Omega)$, $f\ge0$ and $\psi=\Dms f\in\T(\Omega)$. By Lemma \ref{maxprinc-clas}, $\psi\geq 0$ in $\Omega$ and by Lemma \ref{lemma-test} $-\frac{\partial\psi}{\partial\nu}\geq 0$ on $\partial\Omega$. Thus, by \eqref{byparts}, $\int_\Omega uf\geq 0$. Since this is true for every $f\in C^\infty_c(\Omega)$, the result follows.
\end{proof}

\subsection{Proof of Theorem \ref{point}}
Uniqueness is a direct consequence of the comparison principle, Lemma \ref{maxprinc-weak}.
Let us prove that formula \eqref{repr} defines the desired weak solution. Observe that if $u$ is given by \eqref{repr}, then $u\in L^1(\Omega,\delta(x)dx)$. Indeed, 

\begin{multline}\label{L1delta}
\int_\Omega\left|\varphi_1(x) \int_\Omega G^s_\Omega(x,y)d\mu(y)\right|dx
\leq\int_\Omega \int_\Omega G_\Omega^s(x,y)\varphi_1(x)dx\;d|\mu|(y)
\\=\frac1{\lambda_1^s}\int_\Omega\varphi_1(y)\;d|\mu|(y)\le C \Vert \delta\mu\Vert_{\mathcal M(\Omega)}
\end{multline}
This, along with Lemma \ref{harm}, proves that $u\in L^1(\Omega,\delta(x)dx)$ and \eqref{cont}.
Now, pick $\psi\in\T(\Omega)$ and compute, via the Fubini's Theorem, Lemma \ref{inverse} and Lemma \ref{lemma-test},
\begin{multline*}
\int_\Omega u(x)\Ds\psi(x)\;dx = \\
=\ \int_\Omega\left(\int_\Omega G_\Omega^s(x,y)d\mu(y)\right)
\Ds\psi(x)\;dx 
+\int_\Omega\left(\int_{\partial\Omega}P_\Omega^s(x,z)\:d\zeta(z)\right)\Ds\psi(x)\;dx\ = \\
=\ \int_\Omega\left(\int_\Omega G_\Omega^s(x,y)\Ds\psi(x)\:dx\right)d\mu(y) + \int_{\partial\Omega}\left(\int_\Omega P^s_\Omega(x,z)\Ds\psi(x)\;dx\right)d\zeta(z)\ =\\
=\ \int_\Omega\psi(y)\;d\mu(y)-\int_{\partial\Omega}\frac{\partial\psi}{\partial\nu}(z)\;d\zeta(z).
\end{multline*}
\hfill$\square$

\subsection{Proof of Lemma \ref{clas-weak}}

\noindent {\it Proof of 1.}
Consider a sequence $\{\eta_k\}_{k\in\N}\subset C^\infty_c(\Omega)$ of bump functions such that
$0\leq\eta_1\leq\ldots\leq\eta_k\leq\eta_{k+1}\leq\ldots\leq 1$ and 
$\eta_k(x)\uparrow\chi_\Omega(x)$ as $k\uparrow\infty$. 
Consider $\psi\in C^\infty_c(\Omega)$ and define
$f_k:=\eta_k\Ds\psi\in C^\infty_c(\Omega)$, $\psi_k:=\rest^{-s}f_k\in\T(\Omega)$.

Let us first note that the integral 
\[
\int_\Omega u\,\Ds\psi
\]
makes sense in view of \eqref{ineq-lemma1} and \eqref{cont}. The sequence ${\{f_k\}}_{k\in\N}$ trivially converges a.e. to $\Ds\psi$, while
\[
|\psi_k(x)-\psi(x)|\leq\int_\Omega G_\Omega^s(x,y)|\Ds\psi(y)|
\cdot(1-\eta_k(y))\;dy
\]
converges to 0 for any $x\in\Omega$ by dominated convergence.
Since $u$ is a weak solution, it holds
\[
\int_\Omega u\,f_k\ =\ \int_\Omega\psi_k\;d\mu-\int_{\partial\Omega}\frac{\partial\psi_k}{\partial\nu}\;d\zeta.
\]
By dominated convergence $\int_\Omega u\,f_k\to \int_\Omega u\Ds\psi$ and $\int_\Omega\psi_k\:d\mu\to\int_\Omega\psi\:d\mu$.
Indeed for any $k\in\N$,
\[
|f_k|\leq |\Ds\psi|\qquad\hbox{and}\qquad
|\psi_k|\leq\frac1{\lambda_1^s}\left\|\frac{\Ds\psi}{\varphi_1}\right\|_{L^\infty(\Omega)}\varphi_1,
\]
where the latter inequality follows from the maximum principle or the representation formula $\psi_k(x)=\int_\Omega G_\Omega^s(x,y)f_k(y)dy$.
Finally, the convergence 
\[
\int_{\partial\Omega}\frac{\partial\psi_k}{\partial\nu}\;d\zeta\xrightarrow[k\uparrow\infty]{}0
\]
holds by dominated convergence, since
\[
\int_{\partial\Omega}\frac{\partial\psi_k}{\partial\nu}\;d\zeta
=\int_\Omega\mathbb P_\Omega^s\zeta\,f_k
\]
by the Fubini's Theorem, and $\mathbb P_\Omega^s\zeta\in L^1(\Omega,\delta(x)dx)$ while $|f_k|\leq |\Ds\psi|\leq C\delta$
by Lemma \eqref{As-smooth}, so that
\[
\int_\Omega\mathbb P_\Omega^s\zeta\,f_k
\xrightarrow[k\uparrow\infty]{}
\int_\Omega\mathbb P_\Omega^s\zeta\,f\ =\ 0
\]
because $\mathbb P_\Omega^s\zeta$ is $s$-harmonic and $f\in\ C^\infty_c(\Omega)$.

The proof of the boundary trace can be found in Theorems \ref{bound-l1} and \ref{bound-g}, by recalling the representation formula provided by Theorem \ref{point} for the solution to \eqref{prob}.
\hfill$\square$

\noindent {\it Proof of 2.}
Recall that $u$ is represented by
\[
u(x)=\int_\Omega G_\Omega^s(x,y)\,\mu(y)\:dy
+\int_{\partial\Omega}P_\Omega^s(x,y)\,\zeta(y)\:d\sigma(y).
\]
By Point 1. and Lemma \ref{clas-harm}, $u\in C^{2s+\alpha}_{loc}(\Omega)$.
Moreover, $u\in L^1(\Omega,\delta(x)dx)$ thanks to \eqref{cont}. So,
we can pointwisely compute $\Ds u$ by using \eqref{As2} and \eqref{def2}:
this entails by the self-adjointness of the operator in \eqref{def2} that
\[
\int_\Omega \Ds u\,\psi = \int_\Omega u\Ds\psi = \int_\Omega \mu\psi,
\qquad\text{for any }\psi\in C^\infty_c(\Omega)
\]
and we must conclude that $\Ds u=\mu$ a.e.
By continuity the equality holds everywhere.

We turn now to the boundary trace. The contribution given by $\mathbb G_\Omega^s\mu$
is irrelevant, because it is a bounded function as it follows from 
\[
|\mathbb G_\Omega^s\mu(x)|\leq C\|\mu\|_{L^\infty(\Omega)}\int_\Omega\frac{dy}{\left|x-y\right|^{N-2s}}
\]
where we have used \eqref{green-behav}.
Therefore, by Propositions \ref{bound-cont}, there also holds for all $x_0\in\partial\Omega$,
\[
\lim_{x\to x_0, x\in\Omega} \frac{u(x)}{h_1(x)}
= \lim_{x\to x_0, x\in\Omega} \frac{\mathbb{G}_\Omega^s\mu(x)+\mathbb{P}_\Omega^s\zeta(x)}{h_1(x)}
= \zeta(x_0).
\]
\hfill$\square$

\noindent {\it Proof of 3.} By Lemma \ref{u-to-lapl} $\mu\in C^\eps_{loc}(\Omega)$. In addition, we have assumed that $\zeta\in C(\partial\Omega)$. Consider 
\[
v(x)=\int_\Omega G_\Omega^s(x,y)\,\mu(y)\;dy+\int_{\partial\Omega}P_\Omega^s(x,z)\,\zeta(z)\;d\sigma(z)
\]
the weak solution associated to data $\mu$ and $\zeta$. 
By the previous point of the Lemma, $v$ is a classical solution to the equation, 
so that in a pointwise sense it holds
\[
\Ds(u-v)=0\hbox{ in }\Omega,\qquad \frac{u-v}{h_1}=0\hbox{ on }\partial\Omega.
\]
By applying Lemma \ref{maxprinc-clas} we conclude that $|u-v|\leq \eps h_1$ for any $\eps>0$ and thus $u-v\equiv 0$. 
\hfill$\square$

\section{The nonlinear problem}\label{nonlin-sect}

\begin{lem}[Kato's Inequality]\label{kato} For $f\in L^1(\Omega,\delta(x)dx)$ let $w\in L^1(\Omega,\delta(x)dx)$
weakly solve 
\[
\left\lbrace\begin{aligned}
\Ds w &= f & & \hbox{in }\Omega \\
\frac{w}{h_1} &= 0 & & \hbox{on }\partial\Omega.
\end{aligned}\right.
\]
For any convex $\Phi:\R\to\R$, $\Phi\in C^2(\R)$ such that $\Phi(0)=0$
and $\Phi(w)\in L^1_{loc}(\Omega)$, it holds
\[
\Ds\Phi(w)\leq\Phi'(w)\Ds w.
\]
Moreover, the same holds for $\Phi(t)=t^+=t\wedge 0$.
\end{lem}
\begin{proof}
Let us first assume that $f\in C^\alpha_{loc}(\Omega)$.
In this case, by Lemma \ref{clas-harm}, $w\in C^{2s+\alpha}_{loc}(\Omega)$
and the equality $\Ds w=f$ holds in a pointwise sense. Then
\begin{align*}
\Ds\Phi\circ w(x) & =\ \int_\Omega[\Phi(w(x))-\Phi(w(y))]\,J(x,y)\:dy+\kappa(x)\,\Phi(w(x)) \\
& =\ \Phi'(w(x))\int_\Omega[w(x)-w(y)]\,J(x,y)\:dy+\kappa(x)\,\Phi(w(x)) \\
& \ \ -\int_\Omega[w(x)-w(y)]^2\,J(x,y)
\int_0^1\Phi''(w(x)+t[w(y)-w(x)])(1-t)\:dt\:dy \\
&\leq\ \Phi'(w(x))\,\Ds w(x)
\end{align*}
where we have used that $\Phi''\geq 0$ in $\R$ and 
that $\Phi'(t)\leq t\Phi(t)$, which follows from $\Phi(0)=0$.

We deal now with $f\in L^\infty(\Omega)$.
Pick ${\{f_j\}}_{j\in\N}\subseteq C^\infty_c(\Omega)$ converging to $f$
in $L^1(\Omega,\delta(x),dx)$ and bounded in $L^\infty(\Omega)$. 
The corresponding ${\{w_j=\mathbb{G}_\Omega^s f_j\}}_{j\in\N}$
converges to $w$ in $L^1(\Omega,\delta(x)dx)$,
is bounded in $L^\infty(\Omega)$ and without loss of generality
we assume that $f_j\to f$ and $w_j\to w$ a.e. in $\Omega$.
We know that for any $\psi\in \T(\Omega),\ \psi\geq 0$
\[
\int_\Omega\Phi(w_j)\,\Ds\psi\leq\int_\Omega f_j\,\Phi'(w_j)\psi.
\]
By the continuity of $\Phi$ and $\Phi'$ we have $\Phi(w_j)\to\Phi(w)$, $\Phi'(w_j)\to\Phi'(w)$ a.e. in $\Omega$ and that ${\{\Phi(w_j)\}}_{j\in\N}$, ${\{\Phi'(w_j)\}}_{j\in\N}$ are bounded in $L^\infty(\Omega)$. Since ${\{f_j\}}_{j\in\N}$ is converging to $f$ in $L^1(\Omega,\delta(x)dx)$, then 
\[
\int_\Omega\Phi(w_j)\,\Ds\psi\longrightarrow
\int_\Omega\Phi(w)\,\Ds\psi
\qquad\hbox{and}\qquad
\int_\Omega f_j\,\Phi'(w_j)\psi\longrightarrow
\int_\Omega f\,\Phi'(w)\psi
\]
by dominated convergence.

For a general $f\in L^1(\Omega,\delta(x)dx)$ define $f_{j,k}:=(f\wedge j)\vee(-k)$, $j,k\in\N$.
Also, split the expression of $\Phi=\Phi_1-\Phi_2$ into the difference of two increasing function: this can be done in
the following way. The function $\Phi'$ is continuous and increasing in $\R$,
so that it can either have constant sign or there exists $t_0\in\R$ such that $\Phi'(t_0)=0$.
If it has constant sign than $\Phi$ can be increasing or decreasing and 
we can choose respectively $\Phi_1=\Phi,\Phi_2=0$ or $\Phi_1=0,\Phi_2=-\Phi$.
Otherwise we can take
\[
\Phi_1(t)=\left\lbrace\begin{aligned}
& \Phi(t) & t>t_0 \\
& \Phi(t_0) & t\leq t_0
\end{aligned}\right.
\qquad\hbox{and}\qquad
\Phi_2(t)=\left\lbrace\begin{aligned}
& 0 & t>t_0 \\
& \Phi(t_0)-\Phi(t) & t\leq t_0
\end{aligned}\right. .
\]
We already know that
for any $\psi\in \T(\Omega),\ \psi\geq 0$
\[
\int_\Omega\Phi(w_{j,k})\,\Ds\psi\leq\int_\Omega f_{j,k}\,\Phi'(w_{j,k})\psi.
\]
On the right-hand side we can use twice the monotone convergence, 
letting $j\uparrow\infty$ first and then $k\uparrow\infty$.
On the left hand side, by writing $\Phi=\Phi_1-\Phi_2$
again we can exploit several times the monotone convergence by splitting
\begin{multline*}
\int_\Omega\Phi(w_{j,k})\,\Ds\psi = \int_\Omega\Phi_1(w_{j,k})\,[\Ds\psi]^+-\int_\Omega\Phi_1(w_{j,k})\,[\Ds\psi]^-\ +\\
-\ \int_\Omega\Phi_2(w_{j,k})\,[\Ds\psi]^++\int_\Omega\Phi_2(w_{j,k})\,[\Ds\psi]^-
\end{multline*}
to deduce the thesis.

Finally, note that $\Phi(t)=t^+$ can be monotonically
approximated by 
\[
\Phi_j(t)=\frac12\sqrt{t^2+\frac1{j^2}}+\frac t2 -\frac1{2j}
\]
which is convex, $C^2$ and $\Phi_j(0)=0$. So 
\[
\int_\Omega\Phi_j(w)\,\Ds\psi\leq\int_\Omega f\,\Phi_j'(w)\psi.
\]
Since $\Phi_j(t)\uparrow t^+$ and $2\Phi_j'(t)\uparrow 1+\hbox{sgn}(t)=2\chi_{(0,+\infty)}(t)$,
we prove the last statement of the Lemma.
\end{proof}

\begin{theo}\label{exist} 
Let $f(x,t):\Omega\times\R\longrightarrow\R$ be a Carath\'eodory function. 
Assume that there exists a subsolution and a supersolution $\underline u,\overline u\in L^1(\Omega,\delta(x)dx)\cap L^\infty_{loc}(\Omega)$ 
to
\begin{equation}\label{098}
\left\lbrace\begin{aligned}
\Ds u &= f(x,u) & \hbox{ in }\Omega \\
\frac{u}{h_1} &= 0 & \hbox{ on }\partial\Omega
\end{aligned}\right.
\end{equation}
Assume in addition that $f(\cdot,v)\in L^1(\Omega,\delta(x)dx)$ for every $v\in L^1(\Omega,\delta(x)dx)$ such that $\underline u\le v\le \overline u$ a.e.
Then, there exist weak solutions $u_1,u_2\in L^1(\Omega,\delta(x)dx)$ in $[\underline u,\overline u]$ such that any solution in the interval $[\underline u,\overline u]$ satisfies
\[
\underline u\le u_1\le u\le u_2\le \overline u \qquad\text{a.e.}
\]
Moreover, if the nonlinearity $f$ is decreasing in the second variable, then the solution is unique.
\end{theo}
\begin{proof}
According to Montenegro and Ponce \cite{ponce}, the mapping $v\mapsto F(\cdot,v)$, where
\[
F(x,t):=f(x,[t\wedge\overline u(x)]\vee\underline u(x)),\qquad x\in\Omega,t\in\R,
\]
acts continuously from $L^1(\Omega,\delta(x)dx)$ into itself.
In addition, the operator 
\begin{eqnarray*}
\mathcal K:\ L^1(\Omega,\delta(x)dx) & \longrightarrow &  L^1(\Omega,\delta(x)dx) \\
 v(x)\qquad & \longmapsto & \mathcal K(v)(x) = \int_\Omega G_\Omega^s(x,y) F(y,v(y))dy
\end{eqnarray*}
is compact. Indeed, take a bounded sequence ${\{v_n\}}_{n\in\N}$ in $L^1(\Omega,\delta(x)dx)$. On a compact set  $K\subset\subset\Omega$, $\underline u, \overline u$ are essentially bounded and so must be the sequence ${\{F(\cdot,v_n)\}}_{n\in\N}$. By Theorem \ref{point} and Lemma \ref{clas-harm}, ${\{\mathcal K(v_n)\}}_{n\in\N}$ is bounded in $C^\alpha_{loc}(K)\cap L^p(\Omega,\delta(x)dx)$, $p\in[1,(N+1)/(N+1-2s))$. In particular, a subsequence ${\{v_{n_k}\}}_{k\in\N}$ converges locally uniformly to some $v$. By H\"older's inequality, we also have 
\[
\Vert v_{n_k}-v \Vert_{L^1(\Omega\setminus K,\delta(x)dx)} \le
\Vert v_{n_k}-v \Vert_{L^p(\Omega\setminus K,\delta(x)dx)} \Vert \mathds 1_{\Omega\setminus K}\Vert_{L^{p'}(\Omega\setminus K,\delta(x)dx)}.
\]
Hence,
\[
\Vert v_{n_k}-v \Vert_{L^1(\Omega,\delta(x)dx)}\le \Vert v_{n_k}-v \Vert_{L^\infty(K,\delta(x)dx)}\Vert\delta\Vert_{L^1(\Omega)}+C\Vert \mathds 1_{\Omega\setminus K}\Vert_{L^{p'}(\Omega\setminus K,\delta(x)dx}.
\]
Letting $k\to+\infty$ and then $K\to\Omega$, we deduce that $\mathcal K$ is compact and by the Schauder's Fixed Point Theorem, $\mathcal K$ has a fixed point $u\in L^1(\Omega,\delta(x)dx)$. We then may prove that $\underline u\leq u\leq \overline u$ by means of the Kato's Inequality (Lemma \ref{kato}) as it is done in \cite{ponce},
which yields that $u$ is a solution of \eqref{098}.

The proof of the existence of the minimal and a maximal solution $u_1,u_2\in L^1(\Omega,\delta(x)dx)$ can be performed in an analogous way as in \cite{ponce}, as the only needed tool is the Kato's Inequality.

As for the uniqueness, suppose $f$ is decreasing in the second variable and consider two solutions $u,v\in L^1(\Omega,\delta(x)dx)$ to \eqref{098}. By the Kato's Inequality Lemma \ref{kato}, we have
\[
\Ds(u-v)^+\leq \chi_{\{u>v\}}[f(x,u)-f(x,v)]\leq 0
\qquad\hbox{in }\Omega
\]
which implies $(u-v)^+\leq 0$ by the Maximum Principle Lemma \ref{maxprinc-weak}. Reversing the roles of $u$ and $v$, we get
also $(v-u)^+\leq 0$, thus $u\equiv v$ in $\Omega$.
\end{proof}

\subsection{Proof of Theorem \ref{nonhom-cor}}

Problem \eqref{nonhom-prob} is equivalent to
\begin{equation}
\left\lbrace\begin{aligned}
\Ds v &= g(x,\mathbb P_\Omega^s\zeta-v) & \hbox{ in }\Omega \\
\frac{v}{h_1} &= 0 & \hbox{ on }\partial\Omega
\end{aligned}\right.
\end{equation}
that possesses $\overline u=\mathbb P_\Omega^s\zeta$ 
as a supersolution and $\underline u=0$ as a subsolution.
Indeed, by equation \eqref{h1-behav} we have
\[
0\leq\mathbb{P}_\Omega^s\zeta\leq \|\zeta\|_{L^\infty(\Omega)}h_1\leq C\|\zeta\|_{L^\infty(\Omega)}\delta^{-(2-2s)}
\]
Thus any $v\in L^1(\Omega,\delta(x)dx)$ such that $0\leq v\leq\mathbb{P}_\Omega^s\zeta$ satisfies
\[
g(x,v)\leq h(v)\leq h(c\delta^{-(2-2s)})\in L^1(\Omega,\delta(x)dx).
\]
So, all hypotheses of Theorem \ref{exist} are satisfied and the result follows.
\hfill$\square$

\section{Large solutions}\label{extra-sect}

Consider the sequence ${\{u_j\}}_{j\in\N}$ built by solving
\begin{equation}\label{approx}
\left\lbrace\begin{aligned}
\Ds u_j &=\; -u_j^p & \hbox{ in }\Omega \\
\frac{u_j}{h_1} &=\; j & \hbox{ on }\partial\Omega.
\end{aligned}\right.
\end{equation}
Theorem \ref{nonhom-cor} guarantees the existence of such a sequence if
$\delta^{-(2-2s)p}\in L^1(\Omega,\delta(x)dx)$, i.e. $p<1/(1-s)$.
We claim that the sequence is increasing in $\Omega$: indeed the solution
to problem \eqref{approx} is a subsolution for the same problem with boundary datum $j+1$.
In view of this, the sequence $\{u_j\}_{j\in\N}$ admits a pointwise limit, possibly infinite.


\subsection{Construction of a supersolution}

\begin{lem} There exist $\delta_0,C>0$ such that 
\[
\Ds\delta^{-\alpha}\ \geq\ -C\,\delta^{-\alpha p},
\qquad\hbox{for }\delta<\delta_0\hbox{ and }\alpha=\frac{2s}{p-1}.
\]
\end{lem}
\begin{proof}
We use the expression in equation \eqref{As2}. Obviously,
\[
\Ds\delta^{-\alpha}(x)=\int_\Omega[\delta(x)^{-\alpha}-\delta(y)^{-\alpha}]J(x,y)\;dy
+\delta(x)^{-\alpha}\kappa(x) \geq \int_\Omega[\delta(x)^{-\alpha}-\delta(y)^{-\alpha}]J(x,y)\;dy.
\]
For any fixed $x\in\Omega$ close to $\partial\Omega$, split the domain $\Omega$ into three parts:
\begin{align*}
& \Omega_1=\left\lbrace y\in\Omega:\delta(y)\geq\frac32\,\delta(x)\right\rbrace,\\
& \Omega_2=\left\lbrace y\in\Omega:\frac12\,\delta(x)<\delta(y)<\frac32\,\delta(x)\right\rbrace,\\
& \Omega_3=\left\lbrace y\in\Omega:\delta(y)\leq\frac12\,\delta(x)\right\rbrace.
\end{align*}
For $y\in\Omega_1$, since $\delta(y)>\delta(x)$, it holds $\delta(x)^{-\alpha}-\delta(y)^{-\alpha}>0$ and we can drop the integral on $\Omega_1$.
Also, since it holds by equation \eqref{Jbound}
\[
J(x,y)\leq \frac{C}{{|x-y|}^{N+2s}},
\]
the integration on $\Omega_2$ can be performed as in \cite[{\it Second step} in Proposition 6]{a2} providing
\[
PV\int_{\Omega_2}[\delta(x)^{-\alpha}-\delta(y)^{-\alpha}]\,
J(x,y)\:dy\geq -C\,\delta(x)^{-\alpha-2s}=-C\,\delta(x)^{-\alpha p}.
\]
To integrate on $\Omega_3$ we exploit once again \eqref{Jbound} under the form
\[
J(x,y)\leq C\cdot\frac{\delta(x)\,\delta(y)}{{|x-y|}^{N+2+2s}}
\]
to deduce
\[
\int_{\Omega_3}[\delta(x)^{-\alpha}-\delta(y)^{-\alpha}]\,
J(x,y)\:dy\geq -C\,\delta(x)\int_{\Omega_3}\frac{\delta(y)^{1-\alpha}}{{|x-y|}^{N+2+2s}}.
\]
Again, a direct computation as in \cite[{\it Third step} in Prooposition 6]{a2} yields
\[
\int_{\Omega_3}[\delta(x)^{-\alpha}-\delta(y)^{-\alpha}]\,
J(x,y)\:dy\geq -C\,\delta(x)^{-\alpha-2s}=-C\,\delta(x)^{-\alpha p}.
\]
\end{proof}

\begin{lem}\label{prel-supersol}
If a function $v\in L^1(\Omega,\delta(x)dx)$ satisfies
\begin{equation}
\Ds v\in L^\infty_{loc}(\Omega),\qquad
\Ds v(x)\ \geq\ -C\,v(x)^p,\quad \hbox{when }\delta(x)<\delta_0,
\end{equation}
for some $C,\delta_0>0$, then there exists $\overline u\in L^1(\Omega,\delta(x)dx)$
such that
\begin{equation}
\Ds\overline u(x)\ \geq\ -\overline u(x)^p,\qquad \hbox{throughout }\Omega.
\end{equation}
\end{lem}
\begin{proof} Let $\lambda:=C^{1/(p-1)}\vee 1$ and $\Omega_0=\{x\in\Omega:\delta(x)<\delta_0\}$, then
\[
\Ds\left(\lambda v\right)\ \geq\ -\left(\lambda v\right)^p,\qquad \hbox{in }\Omega_0.
\]
Let also $\mu:=\lambda\|\Ds v\|_{L^\infty(\Omega\setminus\Omega_0)}$ and define $\overline u=\mu\mathbb{G}_\Omega^s 1+\lambda v$. On $\overline u$ we have
\[
\Ds\overline u =\mu+\lambda \Ds v\geq \lambda|\Ds v|+\lambda\Ds v\geq -\overline{u}^p
\qquad\hbox{throughout }\Omega.
\]
\end{proof}

\begin{cor}\label{supersol} There exists a function $\overline{u}\in L^1(\Omega,\delta(x)dx)$
such that the inequality
\[
\Ds\overline u\ \geq\ -\overline{u}^p,\qquad \hbox{in }\Omega,
\]
holds in a pointwise sense.
Moreover, $\overline u\asymp\delta^{-2s/(p-1)}$.
\end{cor}
\begin{proof}
Apply Lemma \ref{prel-supersol} with $v=\delta^{-2s/(p-1)}$:
the corresponding $\overline u$ will be of the form
\[
\overline u=\mu\mathbb{G}_\Omega^s1+\lambda\delta^{-2s/(p-1)}.
\]
\end{proof}

\subsection{Existence}

\begin{lem}\label{ujlequ} 
For any $j\in\N$, the solution $u_j$ to problem \eqref{approx} satisfies the upper bound
\[
u_j\ \leq\ \overline{u},\qquad\hbox{in }\Omega,
\]
where $\overline{u}$ is provided by Corollary \ref{supersol}.
\end{lem}
\begin{proof} Write $u_j=jh_1-v_j$ where
\[
\left\lbrace\begin{aligned}
\Ds v_j &=\ \left(jh_1-v_j\right)^p & \hbox{ in }\Omega \\
\frac{v_j}{h_1} &=\ 0 & \hbox{ on }\partial\Omega.
\end{aligned}\right.
\]
and $0\leq v_j\leq jh_1$. Since $\left(jh_1-v_j\right)^p\in L^\infty_{loc}(\Omega)$,
we deduce that $v_j\in C^\alpha_{loc}(\Omega)$ for any $\alpha\in(0,2s)$. 
By bootstrapping $v_j\in C^\infty(\Omega)$ and, by Lemma \ref{clas-harm}, also $u_j\in C^\infty(\Omega)$.
This says that $u_j$ is a classical solution to problem \eqref{approx}.
Now, we have that, 
by the boundary behaviour of $\overline{u}$ stated in Corollary \ref{supersol},
$u_j\leq \overline{u}$ close enough to $\partial\Omega$ (depending on the value of $j$)
and
\[
\Ds\left(\overline{u}-u_j\right)\ \geq\ u_j^p-\overline{u}^p,\qquad\hbox{in }\Omega.
\]
Since $u_j^p-\overline{u}^p\in C(\Omega)$ and
$\lim_{x\to\partial\Omega}u_j^p-\overline{u}^p=-\infty$, there exists $x_0\in\Omega$
such that $u_j(x_0)^p-\overline{u}(x_0)^p=m=:\max_{x\in\Omega}\left(u_j(x)^p-\overline{u}(x)^p\right)$.
If $m>0$ then $\Ds(\overline{u}-u_j)(x_0)\geq m > 0$: this is a contradiction,
as Definition \ref{As2} implies. 
Thus $m\leq 0$ and $u_j\leq \overline{u}$ throughout $\Omega$.
\end{proof}

\begin{theo} For any $\displaystyle p\in\left(1+s,\frac1{1-s}\right)$ 
there exists a function $u\in L^1(\Omega,\delta(x)dx)$
solving 
\[
\left\lbrace\begin{aligned}
\Ds u &=-u^p & \hbox{in }\Omega \\
\delta^{2-2s}u &=+\infty & \hbox{on }\partial\Omega
\end{aligned}\right.
\]
both in a distributional and pointwise sense.
\end{theo}
\begin{proof} 
Consider the sequence $\{u_j\}_{j\in\N}$ provided by problem \ref{approx}:
it is increasing and locally bounded by Lemma \ref{ujlequ},
so it has a pointwise limit $u\leq\overline{u}$, 
where $\overline u$ is the function provided by Corollary \ref{supersol}. 
Since $p>1+s$ and $\overline u \leq C\delta^{-2s/(p-1)}$, then $u\in L^1(\Omega,\delta(x)dx)$. 
Pick now $\psi\in C^\infty_c(\Omega)$, and recall that $\delta^{-1}\Ds\psi\in L^\infty(\Omega)$: 
we have, by dominated convergence,
\[
\int_\Omega u_j\Ds\psi\xrightarrow[j\uparrow\infty]{}\int_\Omega u\Ds\psi,
\qquad \int_\Omega u_j^p\psi\xrightarrow[j\uparrow\infty]{}\int_\Omega u^p\psi
\]
so we deduce
\[
\int_\Omega u\Ds\psi\ =\ -\int_\Omega u^p\psi.
\]
Note now that for any compact $K\subset\subset\Omega$, 
applying Lemma \ref{clas-harm} we get for any $\alpha\in(0,2s)$
\[
\|u_j\|_{C^\alpha(K)}\leq C\left(
\|u_j\|^p_{L^\infty(K)}+\|u_j\|_{L^1(\Omega,\delta(x)dx)}\right)
\leq C\left(
\|\overline u\|^p_{L^\infty(K)}+\|\overline u\|_{L^1(\Omega,\delta(x)dx)}\right)
\]
which means that ${\{u_j\}}_{j\in\N}$ is equibounded and equicontinuous in $C(K)$. By the Ascoli-Arzel\`a Theorem, its pointwise limit $u$ will be in $C(K)$ too. Now, since 
\[
\Ds u\ =\ -u^p\qquad \hbox{in }\mathcal D'(\Omega),
\]
by bootstrapping the interior regularity in Lemma \ref{clas-harm}, we deduce $u\in C^\infty(\Omega)$. So, its spectral fractional Laplacian is pointwise well-defined and the equation is satisfied in a pointwise sense. Also,
\[
\liminf_{x\to\partial\Omega}\frac{u(x)}{h_1(x)}\geq
\liminf_{x\to\partial\Omega}\frac{u_j(x)}{h_1(x)}=j
\]
and we obtain the desired boundary datum.
\end{proof}

\appendix

\section{Appendix}

\subsection{Another representation for the spectral fractional Laplacian}

\begin{lem} For any $u\in H(2s)$ and almost every $x\in\Omega$, there holds
\[\Ds u(x)\ =  PV\int_\Omega[u(x)-u(y)]J(x,y)\;dy+\kappa(x)u(x),\]
where $J(x,y)$ and $\kappa(x)$ are given by \eqref{jkappa}.
\end{lem}
\begin{proof}
Assume that $u=\varphi_j$ is an eigenfunction of the Dirichlet Laplacian associated to the eigenvalue $\lambda_j$. Then, $\Ds u =\lambda_j^s u$,
$e^{t\lapl|_\Omega}u=\int_\Omega p_\Omega(t,\cdot,y)\,u(y)\:dy=e^{-t\lambda_j}u$
and for all $x\in\Omega$
\begin{align}
& \frac{\Gamma(1-s)}s\Ds u(x)\ = \nonumber \\
& =\ \int_0^\infty\left(u(x)-e^{t\left.\lapl\right|_\Omega}u(x)\right)\frac{dt}{t^{1+s}} \label{def2} \\
& =\ \int_0^\infty\left(u(x)-\int_\Omega p_\Omega(t,x,y)u(y)\;dy\right)\frac{dt}{t^{1+s}} \nonumber \\
& =\ \lim_{\epsilon\to0}\int_0^\infty\int_{\Omega\setminus B(x,\epsilon)} p_\Omega(t,x,y)[u(x)-u(y)]\;dy\;\frac{dt}{t^{1+s}}
+\int_0^\infty u(x)\left(1-\int_\Omega p_\Omega(t,x,y)\;dy\right)\;\frac{dt}{t^{1+s}} \nonumber \\
& =\ PV\int_\Omega[u(x)-u(y)]J(x,y)\;dy+\kappa(x)u(x) \label{split}
\end{align}
By linearity, equality holds on the linear span of the eigenvectors. Now, if $u\in H(2s)$, a sequence ${\{u_n\}}_{n\in\N}$ of functions belonging to that span  converges to $u$ in $H(2s)$. In particular, $\Ds u_n$ converges to $\Ds u$ in $L^2(\Omega)$. Note also that for $v\in L^2(\Omega)$,
\[
\frac{s}{\Gamma(1-s)}\left\vert\int_\Omega\int_0^\infty\left(u(x)-e^{t\left.\lapl\right|_\Omega}u(x)\right)\frac{dt}{t^{1+s}}\cdot v(x)\:dx\right\vert
=\left\vert\sum_{k=1}^{+\infty}\lambda_k^s\widehat u_k \widehat v_k\right\vert
\le\Vert u\Vert_{H(2s)}\Vert v\Vert_{L^2(\Omega)}
\]
so that we may also pass to the limit in $L^2(\Omega)$ 
when computing \eqref{def2} along the sequence ${\{u_n\}}_{n\in\N}$. 
By the Fubini's theorem, for almost every $x\in\Omega$, 
all subsequent integrals are convergent 
and the identities remain valid. 
\end{proof}

\subsection{The reduction to the flat case}

In this paragraph we are going to justify the computation
of the asymptotic behaviour of integrals of the type
\[
\int_{A\cap\Omega} F(\delta(x),\delta(y),|x-y|)\;dy \qquad\hbox{as }\ \delta(x)\downarrow 0,
\]
where $A$ is a fixed neighbourhood of $x$ with $A\cap\partial\Omega\neq\emptyset$, by just looking at
\[
\int_0^1dt\int_Bdy'\; F\!\left(\delta(x),t,\sqrt{|y'|^2+|t-\delta(x)|^2}\right).
\]
The first thing to be proved is that
\[
|x-y|^2\asymp|x_0-y_0|^2+|\delta(x)-\delta(y)|^2,
\]
where $x_0,y_0$ are respectively the projections of $x,y$ on $\partial\Omega$.
\begin{lem}
There exists $\eps=\eps(\Omega)>0$ such that
for any $x\in\Omega$, $x=x_0+\delta(x)\grad\delta(x_0)$, $x_0\in\partial\Omega$, with $\delta(x)<\eps$ and any $y\in\Omega$
with $\delta(y)<\eps$ and $|y_0-x_0|<\eps$
\[
\frac12\left(|x_0-y_0|^2+|\delta(x)-\delta(y)|^2\right)\leq
|x-y|^2\leq \frac32\left(|x_0-y_0|^2+|\delta(x)-\delta(y)|^2\right).
\]
\end{lem}
\begin{proof} Call $\Omega_\eps=\{x\in\Omega:\delta(x)<\eps\}$.
Write $x=x_0+\delta(x)\grad\delta(x_0),y=y_0+\delta(y)\grad\delta(y_0)$.
Then
\begin{multline}\label{444}
 |x-y|^2=|x_0-y_0|^2+|\delta(x)-\delta(y)|^2+\delta(y)^2|\grad\delta(x_0)-\grad\delta(y_0)|^2
+2[\delta(x)-\delta(y)]\langle x_0-y_0,\grad\delta(x_0)\rangle\ + \\
+\ 2\delta(y)\langle x_0-y_0,\grad\delta(x_0)-\grad\delta(y_0)\rangle
+2\delta(y)[\delta(x)-\delta(y)]\langle\grad\delta(x_0),\grad\delta(x_0)-\grad\delta(y_0)\rangle.
\end{multline}
Since, for $\eps>0$ small, $\delta\in C^{1,1}(\Omega_\eps)$ and
\begin{align*}
& |\grad\delta(x)-\grad\delta(y)|^2\leq \|\delta\|_{C^{1,1}(\Omega_\eps)}^2|x-y|^2 \\
& \langle x_0-y_0,\grad\delta(x_0)\rangle = O(|x_0-y_0|^2) \\
& |\langle x_0-y_0,\grad\delta(x_0)-\grad\delta(y_0)\rangle| \leq \|\delta\|_{C^{1,1}(\Omega_\eps)}|x_0-y_0|^2 \\
& |\delta(x)-\delta(y)|=\|\delta\|_{C^{1,1}(\Omega_\eps)}|x-y| \\
& |\langle\grad\delta(x_0),\grad\delta(x_0)-\grad\delta(y_0)\rangle|\leq \|\delta\|_{C^{1,1}(\Omega_\eps)}|x_0-y_0|
\end{align*}
The error term we obtained in \eqref{444} can be reabsorbed in the other ones by choosing $\eps>0$ small enough to have
\begin{multline*}
\delta(y)^2|\grad\delta(x_0)-\grad\delta(y_0)|^2+
2|\delta(x)-\delta(y)|\cdot|\langle x_0-y_0,\grad\delta(x_0)\rangle|+2\delta(y)|\langle x_0-y_0,\grad\delta(x_0)-\grad\delta(y_0)\rangle|\ +\\
+\ 2\delta(y)|\delta(x)-\delta(y)|\cdot|\langle\grad\delta(x_0),\grad\delta(x_0)-\grad\delta(y_0)\rangle|\leq\frac12\left(|x_0-y_0|^2+|\delta(x)-\delta(y)|^2\right).
\end{multline*}
\end{proof}

\begin{lem} \label{lemma39} Let $F:(0,+\infty)^3\to(0,+\infty)$ be a continuous function, decreasing in the third variable.
and $\Omega_\eps=\{x\in\Omega:\delta(x)<\eps\}$,
with $\eps=\eps(\Omega)>0$ provided by the previous lemma.
Consider $x=x_0+\delta(x)\grad\delta(x_0)$, $x_0\in\partial\Omega$, and
the neighbourhood $A$ of the point $x$, defined by
$A=\{y\in \Omega_\eps: y=y_0+\delta(y)\grad\delta(y_0),|x_0-y_0|<\eps\}$.
Then there exist
constants $0<c_1<c_2$, $c_1=c_1(\Omega),c_2=c_2(\Omega)$ such that
\begin{multline*}
c_1\int_0^\eps dt\int_{B'_\eps}dy'\; F\!\left(\delta(x),t,c_2\sqrt{|y'|^2+|t-\delta(x)|^2}\right)
\leq \int_A F(\delta(x),\delta(y),|x-y|)\;dy\ \leq \\
\leq\ c_2\int_0^\eps dt\int_{B'_\eps}dy'\; F\!\left(\delta(x),t,c_1\sqrt{|y'|^2+|t-\delta(x)|^2}\right).
\end{multline*}
where the ' superscript denotes objects that live in $\R^{N-1}$.
\end{lem}
\begin{proof} By writing $y=y_0+\delta(y)\grad\delta(y_0),\ y_0\in\partial\Omega$
and using the Fubini's Theorem, we can split the integration into
the variables $y_0$ and $t=\delta(y)$:
\[
\int_A F(\delta(x),\delta(y),|x-y|)\;dy
\leq\int_{B_\eps(x_0)\cap\partial\Omega}\left(\int_0^{\eps} F(\delta(x),t,|x-y_0-t\grad\delta(y_0)|)\;dt \right)d\sigma(y_0).
\]
Using the monotony of $F$ and the above lemma, we get
\[
\int_A F(\delta(x),\delta(y),|x-y|)\;dy\leq
\int_{B_\eps(x_0)\cap\partial\Omega}\left(\int_0^{\eps} F\!\left(\delta(x),t,c\sqrt{|x_0-y_0|^2+|\delta(x)-t|^2}\right)\;dt \right)d\sigma(y_0)
\]
where $c$ is a universal constant.
Representing $B_\eps(x_0)\cap\partial\Omega$ via a diffeomorphism $\gamma$ with a ball $B'_\eps\subset\R^{N-1}$ centered at 0,
we can transform the integration in the $y_0$ variable into the 
integration onto $B'_\eps$. The volume element $|D\gamma|$ will be bounded above and below by
\[
0<c_1\leq |D\gamma|\leq c_2,
\]
in view of the smoothness assumptions on $\partial\Omega$.
\end{proof}

\end{document}